\documentclass[final]{siamltex}

\usepackage{graphicx}
\usepackage{color}
\usepackage{amsmath,amstext,amssymb}
\usepackage{mathrsfs}
\usepackage{multirow}

\newcommand{\trace}{{\mbox{\textrm{Tr}}}}

\newcommand{\bfalpha}{{\mbox{\boldmath $\alpha$}}}
\newcommand{\bfxi}{{\mbox{\boldmath $\xi$}}}

\newcommand{\bfbeta}{{\mbox{\boldmath $\beta$}}}
\newcommand{\bfgamma}{{\mbox{\boldmath $\gamma$}}}

\newcommand{\st}{{\rm s.t.}}
\newcommand{\ba}{\mathbf{a}}
\newcommand{\bA}{\mathbf{A}}
\newcommand{\by}{\mathbf{y}}

\newcommand{\bx}{\mathbf{x}}

\newcommand{\bC}{\mathbf{C}}
\newcommand{\bD}{\mathbf{D}}

\newcommand{\bH}{\mathbf{H}}
\newcommand{\bI}{\mathbf{I}}

\newcommand{\bV}{\mathbf{V}}

\newcommand{\bX}{\mathbf{X}}
\newcommand{\bY}{\mathbf{Y}}

\newcommand{\bw}{\mathbf{w}}
\newcommand{\bB}{\mathbf{B}}

\newcommand{\bh}{\mathbf{h}}
\newcommand{\be}{\mathbf{e}}
\newcommand{\bhw}{\widehat{\mathbf{w}}}

\newcommand{\bhX}{\widehat{\mathbf{X}}}

\newcommand{\bhH}{\widehat{\mathbf{H}}}

\newcommand{\bhx}{\widehat{\mathbf{x}}}

\newcommand{\bbE}{{\mathbb{E}}}

\newcommand{\cI}{\mathcal{I}}
\newcommand{\chI}{\widehat{\mathcal{I}}}
\newcommand{\cO}{\mathcal{O}}

\newcommand{\cM}{\mathcal{M}}
\newcommand{\cN}{\mathcal{N}}

\title{Semidefinite approximation for mixed binary quadratically constrained quadratic
programs\thanks{This research is supported in part by the US AFOSR, grant number FA9550-12-1-0340 and the National Science Foundation, grant number DMS-1015346, and by
the Chinese NSF under the grant 11101261, 11371242 and the First-class Discipline of Universities in Shanghai.
This work was done during a visit by the first author to the Department of Electrical and
Computer Engineering, University of Minnesota.}}

\author{Zi Xu\thanks{Department of Mathematics, College of Sciences, Shanghai University, Shanghai, 200444, China.
({\tt xuzi@shu.edu.cn}). }\and Mingyi Hong\thanks{Department of Electrical and Computer Engineering, University of Minnesota, Minneapolis, MN 55455, USA. ({\tt \{mhong,luozq\}@umn.edu}).}\and Zhi-Quan Luo$^\ddag$}

\begin{document}

\maketitle

\begin{abstract}
Motivated by applications in
wireless communications, this paper develops semidefinite
programming (SDP) relaxation techniques for some mixed binary
quadratically constrained quadratic programs (MBQCQP) and analyzes their
approximation performance. We consider both a minimization and a maximization
model of this problem. For the minimization model, the objective is to
find a minimum norm vector in $N$-dimensional real or complex
Euclidean space, such that $M$ concave quadratic
constraints and a cardinality constraint are satisfied with both binary and continuous variables. By employing a special randomized rounding procedure, we show
that the ratio between the norm of the optimal solution of the minimization
model and its SDP relaxation is upper bounded by $\cO(Q^2(M-Q+1)+M^2)$
in the real case and by $\cO(M(M-Q+1))$
in the complex case. For the
maximization model,  the goal is  to find a maximum norm vector
subject to a set of quadratic constraints and a cardinality
constraint with both binary and continuous variables. We show that
in this case the approximation ratio is bounded from below by $\cO(\epsilon/\ln(M))$ for both the
real and the complex cases. Moreover, this ratio is tight up to a constant factor.
\end{abstract}

\begin{keywords}
nonconvex quadratic constrained quadratic programming, semidefinite programming relaxation,
approximation bound, NP-hard
\end{keywords}

\begin{AMS}
90C22, 90C20, 90C59
\end{AMS}

\pagestyle{myheadings}
\thispagestyle{plain}
\markboth{Z. XU, M. HONG and Z.-Q. LUO}{SEMIDEFINITE APPROXIMATION FOR QCQP}

\section{Introduction}
Motivated by applications in wireless communications, we study in this paper
two classes of mixed binary nonconvex quadratically constrained quadratic programming (MBQCQP)
problems, where the objective functions are quadratic in the continuous
variables and the constraints contain continuous and binary
variables. Although these two classes of optimization problems are nonconvex,
they are amenable to semidefinite programming (SDP) relaxation.
The focus of our study is on the approximation bounds of
the SDP relaxation for both problems.

{\bf The minimization model.}
Consider the following MBQCQP problem:
\begin{align}
\min_{\bw \in \mathbb{F}^{N}, \mbox{\footnotesize\boldmath{$\beta$}}}&\quad\|\bw\|^2\nonumber\\
{\rm s.t.}&\quad \bw^H \bH_i \bw\ge\beta_i\cdot 1+(1-\beta_i)\cdot\epsilon,~i\in \mathcal{M}\label{Generalproblem_in}\\
&\quad \sum_{i\in\mathcal{M}}\beta_i\ge Q,\nonumber\\
&\quad \beta_i\in\{0,1\},~i\in \mathcal{M}\nonumber
\end{align}
where $\mathbb{F}$ is either the field of real numbers $\mathbb{R}$ or the field of complex numbers $\mathbb{C}$, $\mathcal{M}=\{1,\cdots, M\}$, $\bfbeta=(\beta_1, \cdots, \beta_M)^T$, $\bH_i$ $(i=1,\cdots, M)$ are $N\times N$ real symmetric or complex Hermitian positive semidefinite matrices,  $\|\cdot\|$ denotes the Euclidean norm in $\mathbb{F}^N$, $M$ and $Q$ are given integers satisfying $1\leq Q\leq M$, and $\epsilon$ is a given parameter satisfying $0\le \epsilon\le1$.
Throughout, we use the superscript $H$ to denote the complex Hermitian transpose.  Notice that the problem \eqref{Generalproblem_in} can be easily solved either when $N=1$ or $M=1$, by solving a maximum eigenvalue problem. Hence, we shall assume that $N\ge 2$ and $M\ge 2$ in the rest of the paper. We note that problem \eqref{Generalproblem_in} is in general NP-hard, due to the fact that one of its special cases with $Q=M$ is NP-hard; see \cite[Section
2]{Luo07approximationbounds}.

Our interest in problem \eqref{Generalproblem_in} is motivated by its
application in telecommunications. For example,  
consider a cellular network where $M$ users, each equipped with a single
receive antenna, are served by a base station (BS) with $N$ transmit antennas.
Assume that a linear transmit beam $\mathbf{w}\in\mathbb{C}^{N\times 1}$ is
used by the BS to transmit a common message to the users. Let
$\bar{\mathbf{h}}_i\in\mathbb{C}^{N\times 1}$ denote the complex
channel coefficient between the BS and user $i$, and let $n_i$
denote the thermal noise power at the receiver of user $i$. Using
these notations, the signal to noise ratio (SNR) at the receiver of
each user $i\in\mathcal{M}$ can be expressed as ${\rm
SNR}_{i}\triangleq\frac{|\bw^H\bar{\bh}_i|^2}{n_i}$. Such received
SNR level measures the quality of the received signal, and is
directly related to the effective rate of the communication. In
order to successfully decode the transmitted message, typically a
quality of service (QoS) requirement in the form of
$\frac{|\bw^H\bar{\bh}_i|^2}{n_i}\ge\gamma_i$ is imposed by each
user $i\in\cM$, where $\gamma_i$ is a predetermined QoS threshold.
Let us define $\bh_i\triangleq\frac{\bar{\bh}_i}{\sqrt{n_i
\gamma_i}}$ as user $i$'s normalized channel. When assuming that all
$M$ users are served by the BS, the classical physical layer
multicast problem can be formulated to the one that minimizes the
transmit power of the BS while maintaining the QoS requirements
\cite{Sidiropoulos06Multicast}:
\begin{align}
\min_{\bw}&\quad\|\bw\|^2\label{eqMulticast}\\
{\rm s.t.}&\quad {|\bw^H{\bh}_i|^2}\ge1,~i\in\cM\nonumber.
\end{align}
This problem is an NP-hard quadratically constrained quadratic
programming (QCQP) problem; see e.g., \cite{Sidiropoulos06Multicast,
Luo07approximationbounds}. It is a continuous homogeneous QCQP problem, which is a special case of the MBQCQP problem \eqref{Generalproblem_in} when we set $\bH_i=\bh_i\bh_i^H$ and $Q=M$.

However, when the number of users in the network is large, it is
usually not possible to simultaneously guarantee the QoS for all the
users. In this case, user admission control
should be implemented to select a subset of users to serve. For instance,
we can select to serve a subset of $Q$ users (with
$Q$ a given number and $Q\le M$). This results in the following joint physical
layer multicast and admission control problem:
\begin{align}
\min_{\bw\in\mathbb{C}^{N}, \bfbeta}&\quad\|\bw\|^2\label{problemMulticastCardinalityConstraints}\\
{\rm s.t.}&\quad {|\bw^H{\bh}_i|^2}\ge\beta_i,~i\in\cM\nonumber\\
&\quad \sum_{i\in\mathcal{M}}\beta_i\ge Q,\ \
\beta_i\in\{0,1\},~i\in\cM\nonumber.
\end{align}
In the above formulation, the binary variables
$\{\beta_i\}_{i\in\cM}$ indicate whether a particular user $i$
should be served -- when $\beta_i=0$, there is no guarantee that its
QoS constraint will be satisfied. Problem
\eqref{problemMulticastCardinalityConstraints} is closely related to a
different form of the {joint} admission control and
beamforming problem, in which the goal is to pick the {\it maximum
number} of users to serve, subject to their respective QoS
constraints plus the BS's power constraint \cite{Matskani08, Matskani09batch_adaptive, Liuyf12}.
Obviously, \eqref{problemMulticastCardinalityConstraints} is a special case of
the MBQCQP problem \eqref{Generalproblem_in} if we set $\bH_i=\bh_i\bh_i^H$, for
all $i\in\cM$ and $\epsilon=0$.

An effective approach to approximately solve the NP-hard problem
\eqref{eqMulticast} is to use the semidefinite
programming relaxation technique~\cite{luo10SDPMagazine} together with
a randomized rounding step. The idea is to first reformulate
the problem by introducing a rank-1 matrix $\bX=\bw\bw^H$. After
dropping the nonconvex rank-1 constraint on $\bX$, the relaxed
problem becomes an SDP, whose optimal solution ${\bar \bX}$ can be
efficiently computed. A randomization procedure then follows which
converts ${\bar \bX}$ to a feasible solution of \eqref{eqMulticast}. It
has been shown in \cite{Luo07approximationbounds} that such SDP
relaxation scheme generates high-quality solutions, whose worst case
performance bound can be explicitly characterized. Notice that if $\epsilon=1$ or $Q=M$, then problem \eqref{Generalproblem_in} reduces to a continuous homogeneous QCQP problem for which the theoretical ratio between its optimal solution and the SDP relaxation has been shown
to be upper bounded by $27M^2/\pi$ \cite{Luo07approximationbounds}.
However, for the general case when $0\leq\epsilon <1$ with both binary and continuous variables,
there is no known performance guarantees for the performance of SDP
relaxation techniques.

{\bf The maximization model.} Another interesting case of the MBQCQP
problem takes the maximization form as follows:
\begin{align}
\max_{\bw\in \mathbb{F}^{n}, \bfbeta}&\quad\|\bw\|^2\nonumber\\
{\rm s.t.}&\quad \bw^H \bH_i \bw\le\beta_i\cdot\epsilon + (1-\beta_i) \cdot 1,~i\in \cM\label{1.4}\\
&\quad \sum_{i\in\mathcal{M}}\beta_i\ge Q,\  \
 \beta_i\in\{0,1\},~i\in \cM\nonumber
\end{align}
where $0\leq\epsilon \leq 1$ and $1\leq Q\leq M$. The above MBQCQP problem \eqref{1.4}
arises naturally in the interference suppression problem in radar or wireless
communication.  Here, the interference
suppression is captured by the constraints \eqref{1.4},
in which the constants $\epsilon$ and $1$ represent two distinctive
suppression levels. The optimization problem becomes the one that maximizes the gain of the
antenna array while suppressive undesirable interferences. Notice that problem
\eqref{1.4} can be easily solved when $N=1$ or $M=1$. As a result we
assume $N\geq 2$ and $M\geq 2$ in the rest of the paper.

Notice that if $\epsilon=1$, then problem \eqref{1.4} reduces to a
continuous homogeneous QCQP problem for which the theoretical ratio
between its optimal solution and the SDP relaxation has been shown
to be bounded below by $\cO(1/\ln(M))$ \cite{Bental02,Luo07approximationbounds, zhang00DSP}.
However, for the general case when $0\leq\epsilon <1$, no approximation bounds
for SDP relaxation are known.

{\bf Our Contributions.} In Sections $2$ and $3$, we develop two types of SDP relaxations for the optimization models \eqref{problemMulticastCardinalityConstraints} and \eqref{1.4}: One applies SDP relaxation to both the binary and the continuous variables, while the other way simply relaxes binary variables to continuous variables and uses the SDP relaxation for the other continuous variables. Interestingly, we prove that these two types of SDPs are equivalent.
 Given an optimal solution of the relaxed problem, we devise a novel randomization procedure to generate approximate solutions for the original NP-hard MBQCQP problems. 
 Moreover, we analyze the quality of such approximate solutions by
deriving  bounds on the approximation ratios between the optimal solution of
the two MBQCQP problems and those of their corresponding SDP relaxations. Our main results
are as follows: {i)} For the problem \eqref{Generalproblem}, when $\epsilon=0$,
the approximation ratio is upper bounded by $\cO(Q^2(M-Q+1))$ for the real case
and by $\cO(Q(M-Q+1))$ for the complex case; when $0<\epsilon<1$, the approximation ratio is upper bounded by  $\cO(Q^2(M-Q+1)+M^2)$ for the real case and by  $\cO(M(M-Q+1))$ for the complex case; { ii)} For the problem \eqref{1.4}, the ratio can be arbitrarily bad when $\epsilon=0$;
otherwise it is bounded from below by $\cO(\epsilon/\ln(M))$ for both the
real and the complex cases. To the best of knowledge, our analysis is the first
attempt to rigorously characterize the approximation bounds for
MBQCQP-type problems and their SDP relaxations.

{\bf Related Literature.} There is a
sizeable literature on the quality bounds of SDP relaxation for
solving nonconvex QCQP problems, including the works of Luo et al.
\cite{Luo07approximationbounds},  Nemirovski et al.
\cite{Nemirovski99},  A. Ben-Tal et al. \cite{Bental02}. Moreover, for the max-cut problem,
which is a special QCQP problem with only discrete
variables, Goemans and Williamson \cite{Goemans:1995} showed that
the ratio of the optimal value of SDP relaxation over that of the
original problem is bounded below by $0.87856\ldots$. For closely
related results, see \cite{ye01_699, Frieze:1995}. Moreover, there may not be any meaningful worst-case approximation ratio between certain discrete quadratic optimization problems and their semidefinite relaxation (see, e.g., Proposition 3.1 in \cite{So2010} or the discussion in \cite{Kisialiou2010}).
In the absence of the discrete constraints,
Beck and Teboulle \cite{Beck} considered the
continuous nonconvex problem of minimizing the ratio of two
nonconvex quadratic functions over a possibly degenerate ellipsoid,
and showed that the SDP relaxation can return exact solutions under
a certain condition. He et al. \cite{He2008} showed that the ratio
between the optimal value of a homogeneous continuous QCQP problem
and its the SDP relaxation is upper bounded by $\cO (M^2)$ (resp.
$\cO(M)$) in the real (resp. complex) case, if all but one of the
quadratic constraints are convex. However, for the mixed binary and
continuous type QCQP problems, there is no known approximation
bounds for SDP relaxations.

Another popular method to approximately solve the MBQCQP problem
is to relax the 0-1 variables
to continuous variables in the interval $[0,1]$. However, this approach is effective only
when the  resulting continuous relaxation problem is itself convex.
Under this assumption, a variety of algorithms have been proposed,
including branch and bound~\cite{Stubbs},
outer approximation~\cite{Duran1986}, the extended cutting-plane
method~\cite{Westerlund}, and so on. Unfortunately, for the two optimization
models \eqref{Generalproblem} and \eqref{1.4} considered in this paper,
their continuous relaxations are nonconvex.  In this case, there is no known approximation quality analysis results. However, surprisingly, we will show that for our two models, to do the continuous relaxation and the SDP relaxation for the discrete variables are equivalent. Moreover, we obtain some approximation bounds for the SDP relaxation.
In a recent paper, Billionnet et al. \cite{Billionnet2009} proposed a
method called Quadratic Convex Reformulation (QCR) to convert
nonconvex 0-1 QP problems to convex ones by using SDP, rather than
merely relax them. In their follow-up paper, Billionnet et al.
\cite{Billionnet2012} extended the QCR framework from 0-1 QP to mixed-integer
quadratic programming (MIQP).  However, their approach works only
under certain restrictive assumptions, e.g.,  the objective and the constraints
corresponding to the continuous variables are convex, which unfortunately do not
hold in our current context. Some alternative approaches to
MBQCQP have recently appeared in \cite{Burer2011, Saxena10,
Saxena11}. More detailed reviews of recent progress on related problems
can be found in the excellent surveys \cite{Burer2012, Hemmecke, Koppe}.

{\bf Notations.} For a symmetric matrix $\mathbf{X}$,
$\mathbf{X}\succeq 0$ signifies that $\mathbf{X}$ is positive
semi-definite. We use $\trace\left[\mathbf{X}\right]$ and $\bX[i,j]$
to denote the trace and the $(i,j)$th element of a matrix $\bX$, respectively.
For a vector $\bx$, we use $\|\bx\|$ to denote its Euclidean norm,
and use $\bx[i]$ to denote its $i$th element. For a real vector
$\by$, use $\by_{[m]}$ to denote its $m$th largest elements. For a
complex scalar $x$, its complex conjugate is denoted by $\bar{x}$. The notation
$\mathbf{I}_n$ is used to denote a $n\times n$ identity matrix.
Given a set $\mathcal{A}$, $|\mathcal{A}|$ denotes the number of
elements in set $\mathcal{A}$. Also, we use $\mathbb{R}^{N\times M}$ and
$\mathbb{C}^{N\times M}$ to denote the set of real and complex
$N\times M$ matrices, and use $\mathbb{S}^{N}$ and
$\mathbb{S}^{N}_{+}$ to denote the set of $N\times N$ hermitian and
hermitian positive semi-definite matrices, respectively. We use
$\bbE(\cdot)$ to denote the expectation operator. Let $\be_i$ denote the $i$th unit
vector whose entries are all zero except $\be_i[i]=1$, and let $\be$ denote the all $1$
vector. Finally, we use the superscript $H$ and $T$ to denote the complex Hermitian transpose and transpose of a matrix or a vector respectively.
%

\section{Approximation bounds for the minimization model}
\subsection{Reformulation and two SDP relaxations}
In this section we consider the minimization
problem \eqref{Generalproblem_in} and its SDP relaxation.

We first consider the SDP relaxation for both the discrete and the continuous constraints.
Notice that, by monotonicity, we can assume without loss of generality
that the inequality constraint $\sum_{i\in\mathcal{M}}\beta_i\ge Q$ in
\eqref{Generalproblem_in} holds with equality, resulting in the following
equivalent formulation:
\begin{align}
\min_{\bw \in \mathbb{F}^{N}, \mbox{\footnotesize\boldmath$\beta$}}&\quad\|\bw\|^2\nonumber\\
{\rm s.t.}&\quad \bw^H \bH_i \bw\ge\beta_i\cdot 1+(1-\beta_i)\cdot\epsilon,~i\in \mathcal{M}\tag{P1}\label{Generalproblem}\\
&\quad \sum_{i\in\mathcal{M}}\beta_i= Q,\ \
  \beta_i\in\{0,1\},~i\in \mathcal{M}\nonumber
\end{align}
By performing a simple transformation $\alpha_i=2\beta_i-1$, for all
$i\in \mathcal{M}$, we convert the range of the binary variables to
$\{-1,1\}$. We further introduce an auxiliary binary variable
$\ell\in\{-1,1\}$, and transform the problem \eqref{Generalproblem}
equivalently to
\begin{align}
\min_{\bw\in\mathbb{F}^{N}, (\mbox{\footnotesize\boldmath$\alpha$},\ell)}&\quad\|\bw\|^2\nonumber\\
{\st}&\quad\bw^H \bH_i \bw + \frac{1}{4} (\alpha_i - (1-\epsilon)\ell)^2\ge 1+\frac{\epsilon^2}{4},\ i\in \cM\nonumber\\
&\quad \sum_{i=1}^{M}(\alpha_i+\ell)^2= 4Q\nonumber\\
&\quad \ell^2=1, \ \alpha^2_i=1,\ \  i\in\cM\nonumber.
\end{align}
To write the above problem in  a more compact form, we make the
following definitions:
\begin{align}
\bC_i&\triangleq\frac{1}{4}\be_i\be^T_i+\frac{(1-\epsilon)^2}{4}\be_{M+1}\be^T_{M+1}
-\frac{1-\epsilon}{4}\be_i\be^T_{M+1}-\frac{1-\epsilon}{4}\be_{M+1}\be^T_{i}\in\mathbb{R}^{(M+1)\times(M+1)}\nonumber\\
\bA_i&\triangleq\left[ \begin{array}{ll}
\bC_i& \bf{0} \\
\bf{0}&\bH_i
\end{array}\right]\in\mathbb{F}^{(M+1+N)\times(M+1+N)},\ 
\ \bB\triangleq\left[ \begin{array}{lll}
\bI& \be &\bf{0} \\
\be^T & M &\bf{0}\\
\bf{0}& \bf{0}&\bf{0}
\end{array}\right]\in\mathbb{R}^{(M+1+N)\times(M+1+N)}\nonumber\\
\bx^{(1)}&\triangleq[\bfalpha^T, \ell]^T,\ \bx^{(2)}\triangleq\bw,\ \bx\triangleq[\bx^{(1)T}, \bx^{(2)T}]^T\nonumber.
\end{align}

With these definitions, the problem can be written as the following
homogeneous QCQP problem:
\begin{align}
v^{\min}_{{\rm QP}}\triangleq\min_{\bx}&\quad \|\bx^{(2)}\|^2\nonumber\\
{\st}&\quad \bx^H\bA_i\bx\ge 1+\frac{\epsilon^2}{4}, \ i\in\cM\label{problemReformulate}\\
&\quad \bx^H\bB\bx = 4Q,\nonumber\\
&\quad (\bx[i])^2=1, \  i\in\cM\nonumber.
\end{align}

Let ${\bar \bx}$ be a global optimal solution of problem
\eqref{problemReformulate}, and
$v^{\min}_{{\rm QP}}$ be the optimal objective value.

To obtain an approximate solution for the nonconvex quadratic
problem \eqref{problemReformulate}, let us consider its
SDP relaxation. However, caution should be exercised when relaxing
\eqref{problemReformulate}---unlike the conventional QCQP problem
studied in, e.g., \cite{luo10SDPMagazine, Luo07approximationbounds},
the nonconvexity of our current problem arises from {\it both}
continuous \emph{and} binary variables. To proceed, let us introduce a new matrix
variable $\bX\triangleq\bx\bx^H$, which admits the following block structure
\begin{align}
\bX&=\left[ \begin{array}{ll}
\bX^{(1)}&\bX^{(3)}\\
(\bX^{(3)})^H&\bX^{(2)}
\end{array}\right]\in\mathbb{S}^{(M+1+N)\times(M+1+N)}_{+}\label{block}
\end{align}
where $\bX^{(1)}\triangleq\bx^{(1)}(\bx^{(1)})^H$,
$\bX^{(2)}\triangleq\bx^{(2)}(\bx^{(2)})^H$,
$\bX^{(3)}\triangleq\bx^{(1)}(\bx^{(2)})^H$, and the ranks of
matrices $\bX^{(1)}$, $\bX^{(2)}$ and $\bX$ are all $1$.
By dropping all the rank-1 constraints, we obtain the following SDP relaxation
problem for \eqref{problemReformulate}:
\begin{align}
\min&\quad \trace[\bX^{(2)}]\nonumber\\
{\st}&\quad \trace[\bA_i\bX]\ge 1+\frac{\epsilon^2}{4}, \ i\in\cM\label{problemSDR}\\
&\quad \trace[\bB\bX]= 4Q,\nonumber\\
&\quad \bX^{(1)}[i, i]=1, \quad i\in\cM, \quad \bX\succeq 0 \nonumber.
\end{align}
Let ${\tilde \bX}$ denote the optimal solution for this problem, from the block diagonal structure of the matrices $\bC_i,\bA_i,\bB$, it is easy to see that
${\tilde \bX}^{(1)}$ is a {\it real} symmetric PSD matrix, and that without loss of generality we can assume ${\tilde \bX}^{(3)}={\bf 0}$. Moreover, under this assumption, problem \eqref{problemSDR} can be equivalently reformulated as:
\begin{align}
v^{\min}_{{\rm SDP1}}\triangleq\min&\quad \trace[\bX^{(2)}]\nonumber\\
{\st}&\quad \trace \left[\bH_i{\bX}^{(2)}\right]\ge\left[\frac{1+\bX^{(1)}[i,M+1]}{2}\right]+\left[\frac{1-\bX^{(1)}[i,M+1]}{2}\right]\cdot \epsilon, \ i\in\cM\tag{{\rm SDP1}}\label{problemSDR1}\\
&\quad \sum_{i\in\mathcal{M}}\bX^{(1)}[i,M+1]= 2Q-M,\nonumber\\
&\quad \bX^{(1)}[i, i]=1,\ i\in\cM ,\quad \bX^{(1)}\succeq 0,\ \bX^{(2)}\succeq 0\nonumber,
\end{align}
where the variables are $\bX^{(1)}$ and $\bX^{(2)}$, corresponding to the SDP relaxation matrices for the discrete variables and the continuous variables respectively. Let ${\tilde \bX^{(1)}}$ and ${\tilde \bX^{(2)}}$ denote the optimal solution for this problem, and let
$v^{\min}_{{\rm SDP1}}$ denote its optimal objective value.

Alternatively, by relaxing the binary variables to continuous variables in the interval $[0,1]$, while still using SDP relaxation for the continuous variables, we can derive another SDP relaxation problem for \eqref{Generalproblem}:
\begin{align}
v^{\min}_{{\rm SDP2}}\triangleq\min&\quad \trace[\bX^{(2)}]\nonumber\\
{\st}&\quad \trace \left[\bH_i{\bX}^{(2)}\right]\ge \beta_i\cdot 1+(1-\beta_i)\cdot \epsilon, \ i\in\cM\tag{{\rm SDP2}}\label{problemSDR2}\\
&\quad  \sum_{i\in\mathcal{M}}\beta_i= Q,\quad 0\le \beta_i \le 1, \ i\in\cM\nonumber\\
&\quad \bX^{(2)}\succeq 0 \nonumber.
\end{align}
Let $\bar\bfbeta=(\bar{\beta}_1, \cdots, \bar{\beta}_M)^T$ and $\bar \bX^{(2)}$ denote the optimal solution for this problem, and let $v^{\min}_{{\rm SDP2}}$ denote its optimal objective value.
As is well known, SDP relaxation usually yields tighter bounds than the continuous relaxation in many cases. However, as shown in the following lemma,  the two SDP relaxations \eqref{problemSDR1} and \eqref{problemSDR2} for the problem \eqref{Generalproblem} are equivalent.
\begin{lemma}
 Let $(\tilde\bX^{(1)},\tilde\bX^{(2)})$ be an optimal solution of the problem \eqref{problemSDR1}, then $(\tilde{\bfbeta}, \tilde\bX^{(2)})$ is an optimal solution for \eqref{problemSDR2}, with $\tilde{\bfbeta}[i]=\frac{1}{2}+\frac{1}{2}\tilde\bX^{(1)}[i,M+1]$, $\forall \; i\in \cM$. The converse is also true. \label{eqivalentlemma}
\end{lemma}

\begin{proof}
Let $\bfgamma=(\gamma_1, \cdots, \gamma_M)^T$ with $\gamma_i=\frac{1}{2}+\frac{1}{2}\bX^{(1)}[i,M+1]$ ($i\in\cM$) and $\ba=(\bX^{(1)}[1,M+1], \cdots, \bX^{(1)}[M,M+1])^T$, then we have $\ba=2\bfgamma-\be$ and $\bX^{(1)}$ can be written as
\begin{align}
\bX^{(1)}&=\left[ \begin{array}{ll}
\bY& \ba \\
\ba^T&1
\end{array}\right],
\end{align}
where $\bY\in \mathbb{S}^{M\times M}$.
By Schur Complement, we know that $\bX^{(1)}\succeq 0$ is equivalent to $\bY\succeq \ba \ba ^T=(2\bfgamma-\be)(2\bfgamma-\be)^T$. With these definitions and observations, it can be concluded that \eqref{problemSDR1} is equivalent to the following problem:
\begin{align}
\min&\quad \trace[\bX^{(2)}]\nonumber\\
{\st}&\quad \trace \left[\bH_i{\bX}^{(2)}\right]\ge\gamma_i+(1-\gamma_i)\cdot \epsilon, \ i\in\cM\label{reform2.5}\\
&\quad \sum_{i\in\mathcal{M}}\gamma_i= Q,\nonumber\\
&\quad \bY[i, i]=1, i\in\cM ,\quad  \bY\succeq(2\bfgamma-\be)(2\bfgamma-\be)^T , \bX^{(2)}\succeq 0\nonumber.
\end{align}
Assume $(\tilde\bX^{(1)},\tilde\bX^{(2)})$ is an optimal solution for the problem \eqref{problemSDR1} with
\begin{align}
\tilde\bX^{(1)}&=\left[ \begin{array}{ll}
\tilde\bY & 2\tilde\bfbeta-\be \\
(2\tilde\bfbeta-\be)^T&1
\end{array}\right],
\end{align}
with $\tilde\bfbeta=\frac{1}{2}+\frac{1}{2}\tilde\bX^{(1)}[i,M+1]$ ($i\in\cM$), then $(\tilde\bY,  \tilde\bfbeta, \tilde\bX^{(2)})$ will be an optimal solution for \eqref{reform2.5}.
By $\tilde\bY[i, i]=1, i\in\cM$ and $\tilde\bY\succeq(2\tilde\bfbeta-\be)(2\tilde\bfbeta-\be)^T$, we have
$$
-\be\le2\tilde\bfbeta-\be \le \be,
$$
or equivalently $0\le \tilde\bfbeta_i\le 1$, $i\in\cM$. It follows that $(\tilde\bfbeta, \tilde\bX^{(2)})$ is a feasible solution of \eqref{problemSDR2}.
Thus, if $(\bar{\bfbeta}, \bar{\bX}^{(2)})$ is an optimal solution of \eqref{problemSDR2}, then
\begin{align}
\trace[\tilde\bX^{(2)}]\ge \trace[\bar \bX^{(2)}]\label{temp1}.
 \end{align}

Moreover, by the feasibility of $\bar{\bfbeta}$, we have
\begin{align}
-\be\le 2\bar{\bfbeta}-\be\le \be.
 \end{align}
Denote $\Delta\in \mathbb{S}_{+}^{M}$ as a diagonal matrix with its $i$-th diagonal element given by $1-(2\bar\bfbeta[i]-1)^2\ge 0$ , for all $i\in\cM$. Using this definition, we further define a matrix $\tilde{\bY}$ by
 \begin{align}
 \tilde{\bY}=(2\bar{\bfbeta}-\be)(2\bar{\bfbeta}-\be)^T+\Delta.
  \end{align}
 Then, it can be easily checked that $(\tilde{\bY}, \bar{\bX}^{(2)}, \bar{\bfbeta})$ is a feasible solution for \eqref{reform2.5}, and
 \begin{align}
 \trace[\bar \bX^{(2)}]\ge\trace[\tilde\bX^{(2)}]\label{temp2}.
 \end{align}
 By combining\eqref{temp1} and \eqref{temp2}, we have that $ \trace[\bar \bX^{(2)}]=\trace[\tilde\bX^{(2)}]$.
 Moveover, it can also be concluded that $(\tilde\bfbeta, \tilde\bX^{(2)})$ is an optimal solution for \eqref{problemSDR2}.

 The proof of the reverse part is similar and is omitted for space reason.
\end{proof}

By Lemma \ref{eqivalentlemma}, we have $v^{\min}_{{\rm SDP1}}=v^{\min}_{{\rm SDP2}}\le v^{\min}_{{\rm QP}}$. Due to its smaller problem size, the SDP relaxation problem \eqref{problemSDR2} is preferred.

{\bf Remark}: As suggested by an anonymous referee, we can use slightly different technique to derive another SDP relaxation for \eqref{Generalproblem}. Specifically, we first consider the following continuous relaxation of \eqref{Generalproblem}:
\begin{align}
\min &\quad\|\bw\|^2\nonumber\\
{\rm s.t.}&\quad \bw^H \bH_i \bw\ge\alpha_i,~i=1,\cdots, M\label{continue_relax}\\
&\quad \sum_{i=1}^{m}\min\{1, \alpha_i\}= Q+(M-Q)\cdot\epsilon,\nonumber\\
&\quad  \alpha_i\ge \epsilon,~ i=1, \cdots, M.\nonumber
\end{align}
This is a relaxation because any feasible solution $(\bw, \bfbeta)$ to the problem \eqref{Generalproblem} is feasible to problem \eqref{continue_relax} by letting $\bfalpha=\bfbeta+\epsilon(\be -\bfbeta)$. Then, we consider an SDP relaxation to the above continuous problem
\begin{align}
\min &\quad \trace{(\bX)}\nonumber\\
{\rm s.t.}&\quad \trace{(\bH_i \bX)}\ge\alpha_i,~i=1,\cdots, M,\label{SDP_continue_relax}\\
&\quad \sum_{i=1}^{m}\min\{1, \alpha_i\}= Q+(M-Q)\epsilon,\nonumber\\
& \quad \alpha_i\ge \epsilon, ~i=1, \cdots, M, \nonumber\\
& \quad \bX \succeq 0.\nonumber
\end{align}

Although \eqref{SDP_continue_relax} and \eqref{problemSDR2} are derived differently, they are equivalent in the sense that they have the same optimal solution matrix and hence the same optimal objective value. This can be verified easily (we omit the details here). It is also important to note that the two SDP problems have the same problem size and the same number of variables. In the rest of the paper we will use \eqref{problemSDR2} as the basis for our analysis.

 In the following, we aim to generate a feasible solution $\bx$ for
\eqref{Generalproblem} from ${(\bar\bfbeta, \bar \bX^{(2)})}$, and evaluate the quality of
such solution. In particular, we would like to find a constant
$\mu\ge 1$ such that
$$\|\bx^{(2)}\|^2\le \mu v^{\min}_{{\rm SDP2}}.$$
By using the fact that such generated solution is feasible for
problem \eqref{Generalproblem}, we have $v^{\min}_{{\rm QP}}\le
\|\bx^{(2)}\|^2$, which further implies that the same $\mu$ is an
upper bound of the SDP relaxation performance, i.e.,
\begin{align}
v^{\min}_{{\rm QP}}\le \mu v^{\min}_{{\rm SDP2}}\label{defratio}.
\end{align}
The constant $\mu$ will be referred to as the {\it approximation
ratio}.

\subsection{A new randomization procedure}\label{SecNewRand}
Upon obtaining the optimal solution ${(\bar\bfbeta, \bar \bX^{(2)})}$ of problem
\eqref{problemSDR2}, we propose to use the randomization procedure
outlined in Table \ref{tableRandomize} to obtain a feasible solution
for problem \eqref{Generalproblem}.
In essence, this procedure
consists of the following two parts.

{\bf Part 1)} This includes Steps S1 and S2 in Table
\ref{tableRandomize}, in which  we generate the binary variable
$\bx^{(1)}$ from ${\bar \bfbeta}$. It is easy to verify that $\bx^{(1)}$ generated in this
way satisfies the constraints $\sum_{i\in\mathcal{M}}\bx^{(1)}_i= Q$ and $\bx^{(1)}_i\in\{0,1\},~i\in \mathcal{M}$.

{\bf Part 2)} This includes Steps S3 and S4  in Table
\ref{tableRandomize}, where we generate the continuous variable
$\bx^{(2)}$. For all $i\in\cI$, we have $\bx^{(1)}_i=1$ and
\begin{align}
(\bx^{(2)})^H\bH_i\bx^{(2)}&=t^2\bfxi^H\bH_i\bfxi\nonumber\\
&\ge\frac{1}{\bfxi^H\bH_i\bfxi}\bfxi^H\bH_i\bfxi=1,
\end{align}
where the second inequality is due to the definition of $t$. Similarily, for all $i\not\in \cI$, we have $\bx^{(1)}_i=0$ and
\begin{align}
(\bx^{(2)})^H\bH_i\bx^{(2)}=t^2\bfxi^H\bH_i\bfxi\ge\frac{\epsilon}{\bfxi^H\bH_i\bfxi}\bfxi^H\bH_i\bfxi=\epsilon.
\end{align}
In summary, by using the randomization procedure in Table
\ref{tableRandomize}, we obtain a feasible solution $\bx\triangleq[\bx^{(1)T}, \bx^{(2)T}]^T$ to the problem
\eqref{Generalproblem}.

\begin{table}[htb]
\begin{center}
\caption{ The Randomization Procedure for the Minimization Model \eqref{Generalproblem_in}}
\label{tableRandomize} {
\begin{tabular}{|l|}
\hline
\\
S0: Solve the SDP relaxation problem \eqref{problemSDR2} to get ${(\bar{\bfbeta}, \bar \bX^{(2)})}$
with $\bar{\bfbeta}\in\mathbb{R}^{M\times 1}$\\
\qquad and ${\bar \bX}^{(2)}\in\mathbb{S}^{N\times N}_{+}$.\\

S1: Define an index set ${\cI}\triangleq\left\{i:\bar{\bfbeta}[i]\geq \bar{\bfbeta}_{[Q]},~i\in \mathcal{M}\right\}$.\\

S2: Set $\bx^{(1)}[i]=1$ for all $i\in {\cI}$;\   Set $\bx^{(1)}[i]=0$ for all $i\in {\mathcal{M}\setminus \cI}$. \\

S3: Generate a random vector $\bfxi\in \mathbb{F}^n$ from the normal distribution $\cN_\mathbb{F}(\bf{0}, {\bar \bX}^{(2)})$; \\

S4: Let $\bx^{(2)}=t\bfxi$,\\
\qquad with $t=\max\left\{\sqrt{\max_{i\in {\cI}}
\bigg\{\frac{1}{\bfxi^H \mbox{\boldmath $H$}_i\bfxi}\bigg \}}, \sqrt{\max_{i\in {\mathcal{M}\backslash{{\cI}}}}
\bigg\{\frac{\epsilon}{\bfxi^H \mbox{\boldmath $H$}_i\bfxi}\bigg \}}\right\}$.\\

S5: Let $\bx=(\bx^{(1)T}, \bx^{(2)T})^T$.\\

\\ \hline
\end{tabular}}
\end{center}
\end{table}

\subsection{Analysis of the approximation ratio}
In the cases when $Q=M$ or $\epsilon=1$, the problem \eqref{problemReformulate}
reduces to the following continuous QCQP
\begin{align}
\min_{\bw}&\quad\|\bw\|^2\nonumber\\
{\rm s.t.}&\quad {\bw^H{\bH_i}\bw}\ge1,~i\in\cM\nonumber,
\end{align}
for which the SDP relaxation is known to provide a $\frac{27M^2}{\pi}$
approximation in the real case and a $8M$ approximation in the
complex case \cite{Luo07approximationbounds}. In the sequel, we will
only consider the case when $Q\leq M-1$ and $0\le \epsilon<1$.

\subsubsection{The case of $0<\epsilon<1$}

Before presenting our main result, we first need a technical lemma
on the lower bounds of the values of the elements in the set
$\{{\bar \bfbeta}[i]\mid i\in \cM\}$.
\begin{lemma}\label{lemmaXBound}
For $Q\le M-1$, given a constant $\delta$ and define the set
\[\mathcal{R}(\delta)\triangleq\big\{i: {\bar \bfbeta}[i]\ge\delta,\
i=1,\cdots,M\big\}.\] Then we have $|\mathcal{R}(\delta)|\ge Q$ for
all $\delta$ that satisfies $0\le\delta\le \frac{1}{M-Q+1}$.
\end{lemma}
\begin{proof}
We prove this lemma by contradiction. Suppose that
$|\mathcal{R}(\delta)|\le Q-1$, or equivalently
$|\mathcal{M}\setminus\mathcal{R}(\delta)|\ge M-Q+1$. The
feasibility of $\bar \bfbeta$ implies that
\begin{equation}
\sum_{i=1}^{M}{\bar \bfbeta}[i]= Q. \label{eqXSumLowerBound}
\end{equation}
On the other hand, we know that ${\bar \bfbeta}[i]\leq 1$. This, plus the fact that
${\bar \bfbeta}[i]<\delta$ for all $i\not\in R(\delta)$, implies that
\begin{align}
\sum_{i=1}^{M}{\bar \bfbeta}[i]&<(M-Q+1)\delta+Q-1\nonumber\\
&\le (M-Q+1)\cdot \frac{1}{M-Q+1}+Q-1=Q,\nonumber
\end{align}
which contradicts \eqref{eqXSumLowerBound}.
\end{proof}

This result leads to the following characterization of the set $\cI$
generated by the proposed algorithm.

\begin{lemma}\label{lemma:2.2}
After Step S2 in the randomization procedure listed in Table
\ref{tableRandomize}, for all $ i\in \cI$, we have ${\bar \bfbeta}[i]\ge \frac{1}{M-Q+1}.$
\end{lemma}

\begin{proof}
Set $\delta=\frac{1}{M-Q+1}$. Lemma \ref{lemmaXBound} implies
that there exists at least $Q$ elements in $\bar\bfbeta$ that are greater than
$\frac{1}{M-Q+1}$. Since $\cI$ contains the $Q$ largest elements
of $\bar\bfbeta$, the claim follows immediately.
\end{proof}


We are now ready to present a key result of this section, which
essentially bounds the probability that both the normalization
constant and the objective value of the approximate solution are
small. We first consider the real case.


\begin{lemma}
Suppose $Q\leq M-1$ and $0<\epsilon<1$. Let $\bx$ be generated by the randomization
procedure listed in Table \ref{tableRandomize}. Then with
probability $1$, $\bx$ is well defined and feasible for
\eqref{Generalproblem}. Moreover, for any $\alpha>0$ and
$\mu>0$,
\begin{align}
&{\rm Prob}\left( t^2\leq \alpha, \ \|\bx^{(2)}\|^2\leq \mu {\rm Tr\,}[{\bar \bX}^{(2)}]\right)\nonumber\\
&\geq 1-Q\cdot\sigma_1(\alpha)-(M-Q)\cdot \sigma_2(\alpha)-\frac{\alpha}{\mu},\label{lemmainequality}
\end{align}
with
\begin{equation}
\sigma_1(\alpha)=\max\left\{\sqrt{\frac{1}{\alpha c(\epsilon)}}, \frac{2(r-1)}{\pi-2}\cdot \frac{1}{\alpha c(\epsilon)}\right\}, \label{sigma1_R}
\end{equation}
and
\begin{equation}
\sigma_2(\alpha)=\max\left\{\sqrt{\frac{1}{\alpha}}, \frac{2(r-1)}{\pi-2}\cdot \frac{1}{\alpha}\right\}, \mbox{ when $\mathbb{F}=\mathbb{R}$};\label{sigma2_R}
\end{equation}
and with
\begin{equation}
\sigma_1(\alpha)=\max\left\{\frac{4}{3}\cdot\frac{1}{\alpha c(\epsilon)}, 16(r-1)^2\cdot \frac{1}{\alpha^2c(\epsilon)^2}\right\}, \label{sigma1_C}
\end{equation}
and
\begin{equation}
\sigma_2(\alpha)=\max\left\{\frac{4}{3}\cdot\frac{1}{\alpha}, 16(r-1)^2\cdot \frac{1}{\alpha^2}\right\}, \mbox{ when $\mathbb{F}=\mathbb{C}$},\label{sigma2_C}
\end{equation}
where $r:=\min\{\rank ({\bar \bX}^{(2)}), \rank(\bH_i), i\in \cM\}$ and
\begin{equation}\label{eq:ce}
c(\epsilon):=\epsilon+\frac{1-\epsilon}{M-Q+1}.
\end{equation}
\label{lemma:2.3}
\end{lemma}
\begin{proof}
Since the density of
$\bfxi^H\bH_i\bfxi$ is continuous, the
probability $\bfxi^H\bH_i\bfxi=0$ is zero
which implies that $t$ in Step S4 of the randomization procedure is
well defined. The feasibility of $\bx$ generated by the
randomization algorithm is shown in Section \ref{SecNewRand}.

For any $\alpha>0$ and $\mu>0$, we have
\begin{align}
&{\rm Prob}\left( t^2\leq \alpha, \ \|\bx^{(2)}\|^2\leq \mu \trace[{\bar \bX}^{(2)}]\right)\nonumber\\
\geq &\mbox{Prob}\left( t^2\leq \alpha, \ \|\bfxi\|^2\leq \frac{\mu}{\alpha} \trace[{\bar \bX}^{(2)}]\right)\nonumber\\
=&\mbox{Prob}\left(
\frac{1}{\bfxi^H\bH_i\bfxi}\leq \alpha,
\forall\; i\in \cI;\ \frac{\epsilon}{\bfxi^H\bH_i\bfxi}\leq \alpha,
\forall\; i\in {\mathcal{M}\backslash{{\cI}}};\ \|\bfxi\|^2\leq \frac{\mu}{\alpha}
\trace[{\bar \bX}^{(2)}]\right).\label{eqa:2.7}
\end{align}
By Lemma \ref{lemma:2.2}, we have
\begin{eqnarray}
\trace \left[\bH_i{\bar \bX}^{(2)}\right]&\geq& \bar\bfbeta[i]\cdot 1+ (1-\bar\bfbeta[i])\cdot \epsilon\nonumber\\
&=& (1- \epsilon)\cdot \bar\bfbeta[i]+\epsilon\nonumber\\
&\ge&\epsilon+\frac{1-\epsilon}{M-Q+1}:=c(\epsilon),\quad
\forall\; i\in \cI,\label{eqa:2.8}
\end{eqnarray}
where the first inequality follows from the feasibility of  ${(\bar\bfbeta, \bar \bX^{(2)})}$. On the other hand, we have
\begin{align}
\trace \left[\bH_i{\bar \bX}^{(2)}\right]&\geq (1- \epsilon)\cdot \bar\bfbeta[i]+\epsilon \geq \epsilon, \quad  \forall\; i\in {\mathcal{M}\backslash{{\cI}}}.\label{eqa:2.9}
\end{align}
From the way that the random sample $\bfxi$ is generated, we
have $\bbE\left(\bfxi\bfxi^H\right)={\bar \bX}^{(2)}$ and
$\bbE\left(\bfxi^H\bH_i\bfxi\right)=\trace
[\bH_i{\bar \bX}^{(2)}]$. By combining \eqref{eqa:2.7}, \eqref{eqa:2.8} and \eqref{eqa:2.9} , we have

\begin{align}
&{\rm Prob}\left( t^2\leq \alpha, \ \|\bx^{(2)}\|^2\leq \mu
\trace[{\bar \bX}^{(2)}]\right)\nonumber\\
\geq & {\rm Prob} \left( \bfxi^H\bH_i\bfxi\geq \frac{1}{\alpha c(\epsilon)}\trace [\bH_i{\bar \bX}^{(2)}],\ \forall \; i\in \cI;\right.\nonumber\\
&\quad\quad\ \ \left.\bfxi^H\bH_i\bfxi\geq \frac{1}{\alpha}\trace [\bH_i{\bar \bX}^{(2)}],\ \forall\; i\in {\mathcal{M}\setminus \cI}; \|\bfxi\|^2\leq \frac{\mu}{\alpha} \trace[{\bar \bX}^{(2)}]\right)\nonumber\\
\geq & 1-\sum_{i\in \cI} {\rm Prob} \left( \bfxi^H\bH_i\bfxi
\leq \frac{1}{\alpha c(\epsilon)}\trace [\bH_i{\bar \bX}^{(2)}]\right)\nonumber\\
&\ -\sum_{i\in \mathcal{M}\setminus\cI} {\rm Prob} \left( \bfxi^H\bH_i\bfxi
\leq \frac{1}{\alpha}\trace [\bH_i{\bar \bX}^{(2)}]\right)- {\rm Prob}\left( \|\bfxi\|^2> \frac{\mu}{\alpha} \trace[{\bar \bX}^{(2)}]\right)\nonumber\\
&\geq 1-Q\cdot\sigma_1(\alpha)-(M-Q)\cdot \sigma_2(\alpha)-\frac{\alpha}{\mu},\label{finalinequality}
\end{align}
where the last inequality \eqref{finalinequality} is from the
\cite[Lemmas 1 and 3]{Luo07approximationbounds} as well as the
Markov inequality; $\sigma_1(\alpha)$ and $\sigma_2(\alpha)$ are defined as in Lemma \ref{lemma:2.3}.
\end{proof}

Now we can use Lemma~\ref{lemma:2.3} to derive  a fundamental relationship between the
optimal objective value of the problem \eqref{Generalproblem} and the optimal objective value \eqref{problemSDR2} when $\epsilon\not=0$. We first consider the real case.

\begin{theorem}\label{Approximationratio_R}
Let $\mathbb{F}=\mathbb{R}$, $Q\leq M-1$ and $0<\epsilon<1$, we have
\begin{equation*}
v^{\min}_{{\rm QP}}\le
\bar\mu_{\mathbb{R}}\cdot v^{\min}_{{\rm
SDP2}},
\end{equation*}
with
\begin{align}
\bar\mu_{\mathbb{R}} =\max\left\{\frac{27\left [M-Q+\sqrt{\frac{1}{c(\epsilon)}}Q\right]^2}{\pi} ,\frac{12(\sqrt{2M}-1)^2}{(\pi-2)^2c(\epsilon)} \right\},\label{min_ratio_real}
\end{align}
where $c(\epsilon)$ is given by \eqref{eq:ce}.
\end{theorem}

\begin{proof}
By applying a suitable rank reduction procedure if necessary, we can
assume that the rank $\bar{r}$ of optimal SDP solution $\bar {\bX}^{(2)}$
satisfies $\bar{r}(\bar{r}+1)/2\leq M$; cf., \cite{Pataki98, So_Ye_Zhang_2008}.
Moreover, this low rank matrix can be constructed in polynomial
time; see \cite{Huang2007}. Thus, $r$ in \eqref{finalinequality}
satisfies  $r\leq \bar{r} <\sqrt{2M}$. We apply the randomization
procedure listed in Table \ref{tableRandomize} to $\bar {\bX}^{(2)}$. From
\eqref{lemmainequality}, we have
\begin{align}
&{\rm Prob}\left( t^2\leq \alpha, \ \|\bx^{(2)}\|^2\leq \mu \trace[{\bar \bX}^{(2)}]\right)\nonumber\\
&\geq 1-Q\cdot\tilde\sigma_1(\alpha)-(M-Q)\cdot \tilde\sigma_2(\alpha)-\frac{\alpha}{\mu},\label{2.11}
\end{align}
where
\begin{equation}
\tilde{\sigma}_1(\alpha)=\max\left\{\sqrt{\frac{1}{\alpha c(\epsilon)}}, \frac{2(\sqrt{2M}-1)}{\pi-2}\cdot \frac{1}{\alpha c(\epsilon)}\right\}
\end{equation}
and
\begin{equation}
\tilde{\sigma}_2(\alpha)=\max\left\{\sqrt{\frac{1}{\alpha}}, \frac{2(\sqrt{2M}-1)}{\pi-2}\cdot \frac{1}{\alpha}\right\}.
\end{equation}
By setting
\begin{equation}
\alpha=\max\left\{\frac{9\left [M-Q+\sqrt{\frac{1}{c(\epsilon)}}Q\right]^2}{\pi} ,\frac{4(\sqrt{2M}-1)^2}{(\pi-2)^2c(\epsilon)} \right\}\label{2.12}
\end{equation}
and $\mu=3\alpha$,  we have
\begin{equation}
Q\cdot\tilde{\sigma}_1(\alpha)+(M-Q)\cdot \tilde{\sigma}_2(\alpha)=Q\cdot \sqrt{\frac{1}{\alpha c(\epsilon)}}+(M-Q)\cdot \sqrt{\frac{1}{\alpha}}
\le \frac{\sqrt{\pi}}{3}.\label{2.13}
\end{equation}
From \eqref{2.11} and \eqref{2.13}, we have that
\begin{equation}
{\rm Prob}\left( t^2\leq \alpha, \ \|\bx^{(2)}\|^2\leq \mu
\trace[{\bar \bX}^{(2)}]\right)\geq 1-\frac{\sqrt{\pi}}{3}-\frac{1}{3}.
\end{equation}
We see from the above inequality that there is a positive
probability (independent of problem size) of at least
$$1-\frac{\sqrt{\pi}}{3}-\frac{1}{3}=0.0758\ldots$$ that
$$\max\left\{\sqrt{\max_{i\in {\cI}}
\bigg\{\frac{1}{\bfxi^H \mbox{\boldmath $H$}_i\bfxi}\bigg \}}, \sqrt{\max_{i\in {\mathcal{M}\backslash{{\cI}}}}
\bigg\{\frac{\epsilon}{\bfxi^H \mbox{\boldmath $H$}_i\bfxi}\bigg \}}\right\}\leq \alpha,$$ with $\alpha$ defined in \eqref{2.12} and $$\min
\|\bfxi^{(2)}\|^2\leq 3\trace[{\bar \bX}^{(2)}].$$ Let $\bfxi$ be any
vector satisfying these two conditions. Then $\bx$ is feasible for
\eqref{Generalproblem}, so that
\begin{align}
v^{\min}_{{\rm QP}}&\le \|\bx^{(2)}\|^2=\|\bfxi\|^2\cdot \max\left\{\sqrt{\max_{i\in {\cI}}
\bigg\{\frac{1}{\bfxi^H \mbox{\boldmath $H$}_i\bfxi}\bigg \}}, \sqrt{\max_{i\in {\mathcal{M}\backslash{{\cI}}}}
\bigg\{\frac{\epsilon}{\bfxi^H \mbox{\boldmath $H$}_i\bfxi}\bigg \}}\right\}\nonumber\\
&\leq \alpha\cdot 3\trace[{\bar \bX}^{(2)}]=\bar\mu_{\mathbb{R}} \cdot  v^{\min}_{{\rm SDP2}},
\end{align}
where the last equality uses $\trace[{\bar \bX}^{(2)}]=v^{\min}_{{\rm
SDP}}$ and $\bar\mu_{\mathbb{R}}$ is defined as in \eqref{min_ratio_real}.
\end{proof}

It is important to note that the constant $\bar\mu_{\mathbb{R}}$ derived above is in the order of $\cO(Q^2(M-Q+1)+M^2)$ in the worst case. Also note that, when $\epsilon=1$ or $M=Q$, Theorem \ref{Approximationratio_R} still holds since the proof is the same. In both cases, we have $\bar\mu_{\mathbb{R}}=\frac{27M^2}{\pi}$, which is exactly the same as the result stated in \cite{Luo07approximationbounds}.
%

We then consider the complex case.
\begin{theorem}\label{thmApproximationStep2complex}
Let $\mathbb{F}=\mathbb{C}$, $Q\leq M-1$ and $0<\epsilon<1$. We have
$v^{\min}_{{\rm QP}}\le \bar\mu_{\mathbb{C}}\cdot v^{\min}_{{\rm SDP2}}$ with
\begin{align}
\bar\mu_{\mathbb{C}} =\max\left\{8\left [M-Q+\frac{1}{c(\epsilon)}Q\right],\frac{24(\sqrt{M}-1)^2}{c(\epsilon)} \right\},\label{min_ratio_com}
\end{align}
where $c(\epsilon)$ is defined by \eqref{eq:ce}.
\end{theorem}
\begin{proof}
Following the similar steps in the proof of Theorem
\ref{Approximationratio_R}, in this case we have
\begin{align}
&{\rm Prob}\left( t^2\leq \alpha, \ \|\bx^{(2)}\|^2\leq \mu \trace[{\bar \bX}^{(2)}]\right)\nonumber\\
&\geq 1-Q\cdot\bar\sigma_1(\alpha)-(M-Q)\cdot \bar\sigma_2(\alpha)-\frac{\alpha}{\mu},
\end{align}
where
\begin{equation*}
\bar\sigma_1(\alpha)=\max\left\{\frac{4}{3}\cdot\frac{1}{\alpha\cdot c(\epsilon)}, 16(\sqrt{M}-1)^2\cdot \frac{1}{\alpha^2\cdot c(\epsilon)^2}\right\},
\end{equation*}
and
\begin{equation*}
\bar\sigma_2(\alpha)=\max\left\{\frac{4}{3}\cdot\frac{1}{\alpha}, 16(\sqrt{M}-1)^2\cdot \frac{1}{\alpha^2}\right\}.
\end{equation*}
We choose
\begin{equation}
\alpha=\max\left\{4\left [M-Q+\frac{1}{c(\epsilon)}Q\right], \frac{12(\sqrt{M}-1)^2}{c(\epsilon)} \right\},\quad \mu=2\alpha.\label{2.22}
\end{equation}
By $\alpha\ge \frac{12(\sqrt{M}-1)^2}{c(\epsilon)},$ we can easily verify that
$$\frac{4}{3}\cdot\frac{1}{\alpha\cdot c(\epsilon)}\ge 16(\sqrt{M}-1)^2\cdot \frac{1}{\alpha^2\cdot c(\epsilon)^2},$$
and
$$\frac{4}{3}\cdot\frac{1}{\alpha}\ge 16(\sqrt{M}-1)^2\cdot \frac{1}{\alpha^2}.$$
Then, by  $\alpha\ge 4\left [M-Q+\frac{1}{c(\epsilon)}Q\right],$ we have
\begin{equation}
Q\cdot\bar\sigma_1(\alpha)+(M-Q)\cdot \bar\sigma_2(\alpha)=\frac{4}{3\alpha}\cdot\left( \frac{Q}{c(\epsilon)}+M-Q\right)\le \frac{1}{3}.\label{2.23}
\end{equation}
By \eqref{2.22} and \eqref{2.23}, we have that
\begin{align}
&{\rm Prob}\left( t^2\leq \alpha, \ \|\bx^{(2)}\|^2\leq \mu \trace[{\bar \bX}^{(2)}]\right)\nonumber\\
&\geq 1-\frac{1}{3}-\frac{1}{2}=\frac{1}{6}.
\end{align}
We see from the above inequality that there is a positive probability (independent of problem size) of at least $1/6$
that
$$\max\left\{\sqrt{\max_{i\in {\cI}}
\bigg\{\frac{1}{\bfxi^H \mbox{\boldmath $H$}_i\bfxi}\bigg \}}, \sqrt{\max_{i\in {\mathcal{M}\backslash{{\cI}}}}
\bigg\{\frac{\epsilon}{\bfxi^H \mbox{\boldmath $H$}_i\bfxi}\bigg \}}\right\}\leq \alpha,$$
and $$\min \|\bfxi^{(2)}\|^2\leq 2\trace[{\bar \bX}^{(2)}].$$
Let $\bfxi$ be any vector satisfying these two conditions.
Then $\bx$ is feasible for \eqref{Generalproblem}, so that
\begin{align}
v^{\min}_{{\rm QP}}&\le \|\bx^{(2)}\|^2=\|\bfxi^{(2)}\|^2\cdot \max\left\{\sqrt{\max_{i\in {\cI}}
\bigg\{\frac{1}{\bfxi^H \mbox{\boldmath $H$}_i\bfxi}\bigg \}}, \sqrt{\max_{i\in {\mathcal{M}\backslash{{\cI}}}}
\bigg\{\frac{\epsilon}{\bfxi^H \mbox{\boldmath $H$}_i\bfxi}\bigg \}}\right\}\nonumber\\
&\leq \alpha\cdot 2\trace[{\bar \bX}^{(2)}]=\bar\mu_{\mathbb{C}}\cdot v^{\min}_{{\rm SDP2}},
\end{align}
where the last equality uses $\trace[{\bar \bX}^{(2)}]=v^{\min}_{{\rm SDP2}}$ and $\bar\mu_{\mathbb{C}}$ is defined in \eqref{min_ratio_com}.
\end{proof}

We note here that the constant $\bar\mu_{\mathbb{C}}$ derived is in the order of $\cO(M(M-Q+1))$ in the worst case.

{\bf Remark.} Note that, when $\epsilon=1$ or $M=Q$, Theorem \ref{thmApproximationStep2complex} still holds since the proof is the same. In both cases, we have
\begin{equation}
\bar\mu_{\mathbb{C}}=\max\{8M, 24(\sqrt{M}-1)^2\},\label{newratio_complex}
\end{equation}
  which is not exactly the same as the result $8M$ stated in \cite{Luo07approximationbounds}, although the order is the same. The reason is as follows.
In the proof of  Theorem 2 in \cite{Luo07approximationbounds}, by choosing $\gamma=\frac{1}{4M}$ and using $r\le \sqrt{M}$, a key inequality that
\begin{equation}
\frac{4}{3}\gamma \ge 16 (r-1)^2 \gamma ^2, \quad \forall\; M=1, 2,\cdots \label{inequalityinluo}
\end{equation}
is used to get the final result. It can be verified that the inequality
\begin{align}
M\ge 3(\sqrt{M}-1)^2, \quad \forall\; M=1, 2,\cdots \label{eq:inequalitysufficient}
\end{align}
is needed to guarantee \eqref{inequalityinluo}.
If such inequality is true, then clearly we have: $\max\{8M, 24(\sqrt{M}-1)^2\}=8M$.  However, it can be shown that \eqref{eq:inequalitysufficient} is not true when $M\ge 6$. Thus, the approximation ratio for the complex case when $Q=M$ should be written as: $\max\{8M, 24(\sqrt{M}-1)^2\}$.


\subsubsection{The case for $\epsilon=0$} In this special case, the approximation ratio obtained is in fact a little better than the case with $\epsilon\ne 0$. In the following, we will first state our result, and then provide two ways of proving it.

\begin{theorem}\label{thm_epsilon0}
Assume $Q\leq M-1$ and $\epsilon=0$. We have that
\begin{equation}
v^{\min}_{{\rm QP}}\le \frac{27Q^2(M-Q+1)}{\pi}\cdot v^{\min}_{{\rm SDP2}}, \mbox{ when $\mathbb{F}=\mathbb{R}$},
\end{equation}
and
\begin{align}
v^{\min}_{{\rm QP}}\le \max\{8Q, 24(\sqrt{Q}-1)^2\}\cdot(M-Q+1)\cdot v^{\min}_{{\rm SDP2}}, \mbox{ when $\mathbb{F}=\mathbb{C}$}.
\end{align}
\end{theorem}
{\it Outline of the first proof}: The first proof follows similar steps as the case of $\epsilon\not=0$.
Firstly, for any $\alpha>0$ and
$\mu>0$, similarly to Lemma \ref{lemma:2.3}, we can prove that
\begin{align*}
&{\rm Prob}\left( t^2\leq \alpha, \ \|\bx^{(2)}\|^2\leq \mu \trace[{\bar \bX}^{(2)}]\right)\nonumber\\
&\geq 1-Q\cdot\tilde{\sigma}(\alpha)-\frac{\alpha}{\mu},
\end{align*}
with
\begin{equation}
\tilde{\sigma}(\alpha)=\max\left\{\sqrt{\frac{M-Q+1}{\alpha}}, \frac{2(r-1)}{\pi-2}\cdot \frac{M-Q+1}{\alpha}\right\}, \mbox{when $\mathbb{F}=\mathbb{R}$},
\end{equation}
and
\begin{equation}
\sigma(\alpha)=\max\left\{\frac{4}{3}\cdot\frac{M-Q+1}{\alpha}, 16(r-1)^2\cdot \frac{(M-Q+1)^2}{\alpha^2}\right\}, \mbox{when $\mathbb{F}=\mathbb{C}$},
\end{equation}
where $r:=\min\{\rank ({\bar \bX}^{(2)}), \rank(\bH_i), i\in \cM\}$. Note that, after fixing $\cI$, without loss of generality, we can assume that $r<\sqrt{2Q}$ (resp. $r<\sqrt{Q}$) in the real (resp.\ complex) case by rank reduction procedure. Then, we can derive the approximation bound by setting $\alpha=\frac{9Q^2(M-Q+1)}{\pi}$ (resp.\ $\alpha=(M-Q+1)\cdot\max\{4Q, 12(\sqrt{Q}-1)^2\}$), $\mu=3\alpha$ (resp. $\mu=2\alpha$) in the real (resp. complex) case. $\Box$

{\it{Outline of the second proof\,}}\footnote{We thank the anonymous referee for suggesting this simpler proof for the case when $\epsilon=0$. }: Let $\tilde\bX=(M-Q+1)\bar\bX^{(2)}$. By Lemma \ref{lemma:2.2}, it is clear that
$$\trace(\bH_i\tilde\bX)=(M-Q+1)\trace(\bH_i\bar\bX^{(2)})\ge (M-Q+1)\cdot \beta_i\ge 1,\ \forall\; i\in  {\cI}.$$
Therefore $\tilde{\bX}$ is a feasible solution to (SDP2).

Note that $|\cI|=Q$. After fixing $\cI$, without loss of generality, we assume that $\rank (\tilde\bX) \leq \sqrt{2Q}$ (resp.\ $\rank (\tilde\bX) \leq \sqrt{Q}$) by rank reduction procedure in the real (resp.\ complex) case. By using the rounding method based on $\bar\bX^{(2)}$ (or equivalently $\tilde\bX$) listed in Table \ref{tableRandomize} to get a feasible solution for \eqref{Generalproblem}.
From \cite[Theorems~1~and~2]{Luo07approximationbounds}, one may find $\tilde\bx$ such that
$$\tilde\bx^H\bH_i\tilde\bx\ge 1, \ \forall\; i\in  {\cI}.$$
Moreover, when $\mathbb{F}=\mathbb{R}$, we have
$$\|\tilde\bx\|^2\le \frac{27Q^2}{\pi}\cdot \trace(\tilde\bX)= \frac{27Q^2(M-Q+1)}{\pi}\cdot \trace(\bar\bX^{(2)}),$$
and when $\mathbb{F}=\mathbb{C}$,
$$\|\tilde\bx\|^2\le \max\{8Q, 24(\sqrt{Q}-1)^2\}\cdot\trace(\tilde\bX)= \max\{8Q, 24(\sqrt{Q}-1)^2\}\cdot(M-Q+1)\cdot \trace(\bar\bX^{(2)}).$$
The above inequality is equivalent to \eqref{newratio_complex}. This completes the proof of Theorem~\ref{thm_epsilon0}. $\Box$

However, this proof technique can not be generalized to the case $0<\epsilon<1$. The reason is that after step S1 in Table \ref{tableRandomize}, it is not clear how to appropriately scale the solution of (SDP2) to obtain the ratio stated in Theorem \ref{Approximationratio_R}--\ref{thmApproximationStep2complex}.

We summarize the approximation bounds for the
considered minimization model \eqref{Generalproblem} in Table 2.2. 
\begin{table}[htb]
\begin{center}
\caption{ The Approximation Ratio for the Minimization Model}
\begin{tabular}{|c|c|c|c|}\hline \multicolumn{2}{|c|}{$ $} & &\\
\multicolumn{2}{|c|}{$ Parameters$} & $\mathbb{F}=\mathbb{R}$ & $\mathbb{F}=\mathbb{C}$ \\\hline
 \multicolumn{2}{|c|}{$ $}  & &\\
\multicolumn{2}{|c|}{$ 0<\epsilon<1$} & $\max\left\{\frac{27\left [M-Q+\sqrt{\frac{1}{c(\epsilon)}}Q\right]^2}{\pi} \right.$, & $\max\left\{\frac{24(\sqrt{M}-1)^2}{c(\epsilon)},\qquad\quad\right.$ \\
\multicolumn{2}{|l|}{$ $}  & \quad$\frac{12(\sqrt{2M}-1)^2}{(\pi-2)^2c(\epsilon)} \Bigg\}.$&$\left.8\left [M-Q+\frac{1}{c(\epsilon)}Q\right]\right\}$\\ \multicolumn{2}{|l|}{$ $} & &\\\hline
  & & &\\
&$Q=M$ & $\frac{27M^2}{\pi}$  & $\max\{8M, 24(\sqrt{M}-1)^2\}$  \\ \multirow{4}{*}{$\epsilon=0$}& &  (see \cite{Luo07approximationbounds}) & (see \cite{Luo07approximationbounds} and \eqref{newratio_complex}) \\   & & &\\ \cline{2-4} &&&\\
&$Q\le M-1$ & $\frac{27Q^2(M-Q+1)}{\pi}$ & $\max\left\{8Q, 24(\sqrt{Q}-1)^2\right\}$\\ & & &$\cdot(M-Q+1)$\\  & & &\\\hline
\end{tabular}
\end{center}
\label{ratio_compare}
\end{table}


\section{Approximation bounds for the maximization model}
\subsection{The SDP Relaxation}
By monotonicity, we can assume without loss of generality
that the inequality constraint $\sum_{i\in\mathcal{M}}\beta_i\ge Q$ in
\eqref{1.4} holds with equality, resulting in the following
equivalent formulation:
\begin{align}
v^{\max}_{{\rm QP}}\triangleq\max_{\bw \in \mathbb{F}^{N},\mbox{\footnotesize\boldmath $\beta$}}&\quad\|\bw\|^2\nonumber\\
{\rm s.t.}&\quad \bw^H \bH_i \bw\le\beta_i\cdot \epsilon+(1-\beta_i)\cdot 1,~i\in \mathcal{M}\tag{P2}\label{P2}\\
&\quad \sum_{i\in\mathcal{M}}\beta_i= Q,\ \
  \beta_i\in\{0,1\},~i\in \mathcal{M}\nonumber
\end{align}
Define the global optimal solution for problem \eqref{P2} as ${\bar \bx}$, and its objective value as $v^{\max}_{{\rm QP}}$. Similar to the minimization model considered in Section~2,
by relaxing the binary variables to continuous variables in the interval $[0,1]$, while using SDP relaxation for the continuous variables at the same time, we obtain the SDP relaxation problem for \eqref{P2}:
\begin{align}
v^{\max}_{{\rm SDP}}\triangleq\max&\quad \trace[\bX^{(2)}]\nonumber\\
{\st}&\quad \trace \left[\bH_i{\bX}^{(2)}\right]\le \beta_i\cdot \epsilon+(1-\beta_i)\cdot 1, \ i\in\cM\tag{{\rm SDP3}}\label{problemSDP}\\
&\quad  \sum_{i\in\mathcal{M}}\beta_i= Q,\quad 0\le \beta_i \le 1, \ i\in\cM\nonumber\\
&\quad \bX^{(2)}\succeq 0 \nonumber.
\end{align}
Similar to Lemma \ref{eqivalentlemma}, we can prove that problem \eqref{problemSDP} is equivalent to the SDP relaxation problem for both the discrete and the continuous constraints. Let $(\hat\bfbeta,\bhX^{(2)})$ denote the optimal solution for this problem,  and
$v^{\max}_{{\rm SDP}}$  denote its optimal objective value. Obviously, we have $v^{\max}_{{\rm SDP}}\ge v^{\max}_{{\rm QP}}$.

\subsection{Approximation ratio}\label{Section:3.2}

In this section, we consider using SDP relaxation to
approximately solve the NP-hard problem \eqref{P2}. In particular,
we would like to find a finite constant $0<\hat{\mu}<1$ such that
$v^{\max}_{{\rm QP}}\ge\hat{\mu} v^{\max}_{{\rm SDP}}$. We refer this $\hat{\mu}$ as
 the {\it approximation ratio} for the maximization model. The larger the $\hat{\mu}$, the tighter the SDP relaxation.

First, we ask whether for all values of
$\epsilon\in[0,1)$, SDP relaxation for the problem \eqref{P2}
can provide an approximately optimal solution with an objective value
that is within a constant factor to the maximum value of \eqref{P2}.
Unfortunately the answer to this  question is negative. That is, the
solution obtained by the SDP relaxation can be arbitrarily bad in
terms of approximation ratio, for certain value of $\epsilon$. Below
we show through an example that if $\epsilon=0$, the ratio between
$v^{\max}_{{\rm QP}}$ and $v^{\max}_{{\rm SDP}}$ can be zero.

{\bf Example 3.1}: Consider\begin{align}
\max&\quad\|\bw\|^2\nonumber\\
{\rm s.t.}&\quad \bw^H\bw\le 1-\beta_i,~i=1,2\label{problemExample}\\
&\quad \beta_1+\beta_2= 1\nonumber\\
&\quad \beta_i\in\{0,1\},~i=1,2,\nonumber
\end{align}
where $\bw\in\mathbb{R}^2$. Its SDP relaxation is
\begin{align}
\max&\quad\trace[\bX^{(2)}]\nonumber\\
{\rm s.t.}&\quad \trace[\bX^{(2)}]\le 1-\beta_i,~i=1,2\label{problemExampleSDP}\\
&\quad \beta_1+\beta_2=1, \ 0\le \beta_i\le 1,~i=1,2\nonumber\\
&\quad \bX^{(2)}\succeq 0 \nonumber.
\end{align}
For problem \eqref{problemExample}, $\bar{\bw}^*=[0,0]^T$ is its optimal
solution, and we have $v^{\max}_{{\rm QP}}=0$. On the other hand, it
can be easily checked that
$$\tilde{\bX}^{(2)}=\left[ \begin{array}{cc}
\frac{1}{4} & 0\\
0 & \frac{1}{4}
\end{array}\right]$$
is a feasible solution for problem \eqref{problemExampleSDP} with $\beta_1=\beta_2=\frac{1}{2}$, so
$v^{\max}_{{\rm SDP}}\geq \trace[\tilde{\bX}^{(2)}]=\frac{1}{2}$.
Therefore, for this example the ratio $v^{\max}_{{\rm
QP}}/v^{\max}_{{\rm SDP}}$ is zero. \hfill $\square$

\begin{table}[htb]
\begin{center}
\caption{The Randomization Procedure for the Maximization Model \eqref{1.4}}
\label{randomprocedure_2} {
\begin{tabular}{|l|}
\hline\\
S0: Obtain the optimal solution $(\hat\bfbeta, \bhX^{(2)})$ of \eqref{problemSDP}, with $\hat\bfbeta\in\mathbb{R}^{M\times 1}$\\
\qquad and ${\bhX}^{(2)}\in\mathbb{S}^{N\times N}_{+}$.\\

S1: Define an index set ${\chI}\triangleq\left\{i:\hat\bfbeta[i]\geq \hat\bfbeta_{[Q]},~i\in \mathcal{M}\right\}$.\\

S2: Set $\bhx^{(1)}[i]=1$ for all $i\in {\chI}$;\ Set $\bhx^{(1)}[i]=0$ for all $i\in {\mathcal{M}\setminus {{\chI}}}$. \\

S3: Generate a random vector $\widehat{\bfxi}\in \mathbb{F}^N$ from the normal\\ \qquad distribution $\cN_\mathbb{F}(\bf{0}, \bhX^{(2)})$. \\

S4: Let $\bhx^{(2)}=\hat{t}\widehat{\bfxi}$, where\\
\qquad $\hat{t}=\min\left\{\sqrt{\min_{i\in {\chI}}
\bigg\{\frac{\epsilon}{\widehat{\bfxi}^H\mbox{\boldmath $H$}_i\widehat{\bfxi}}\bigg\}}, \sqrt{\min_{i\in {\mathcal{M}\backslash{{\chI}}}}\bigg\{\frac{1}{\hat{\bfxi}^H\mbox{\boldmath $H$}_i\hat{\bfxi}}\bigg\}}\right\}$.\\

S5: Let $\bhx=(\bhx^{(1)T}, \bhx^{(2)T})^T$.\\
\\ \hline
\end{tabular}}
\end{center}
\end{table}

In light of the above example, we will focus on the
case $\epsilon\in(0,1)$ and analyze the
SDP approximation ratio for problem \eqref{P2} in this case. To this end, we first propose a
randomization procedure to convert $(\hat\bfbeta, \bhX^{(2)})$ to a feasible solution of
the original problem \eqref{P2}. The detailed steps are listed in
Table \ref{randomprocedure_2}.

By the same argument as in Section \ref{SecNewRand}, we can verify that
$\bhx^{(1)}$ generated from S1-S2 of Table \ref{randomprocedure_2}
satisfies the constraints $\sum_{i\in\mathcal{M}}\bhx[i]= Q$ and $\bhx[i]\in\{0,1\}$
for all $i=1,\cdots, M$. To argue that $\bhx^{(2)}$ is also
feasible, it remains to show that $(\bhx^{(2)})^H \bH_i \bhx^{(2)}\le\bhx[i]\cdot\epsilon + (1-\bhx[i]) \cdot 1 $ for all $i\in\cM$. To see this, note that
for all $i\in\chI$, we use $\bhx^{(1)}[i]=1$ to obtain
\begin{align}
(\bhx^{(2)})^H\bH_i\bhx^{(2)}=\hat{t}^2\widehat{\bfxi}^H\bH_i\widehat{\bfxi}
\le \frac{\epsilon}{\hat{\bfxi}^H\bH_i\hat{\bfxi}}\hat{\bfxi}^H\bH_i\hat{\bfxi}
=\epsilon.\label{3.4}
\end{align}
Similarly, for all $i\in {\mathcal{M}\backslash{{\chI}}}$,  we have $\bhx^{(1)}[i]=0$ and thus
\begin{align}
(\bhx^{(2)})^H\bH_i\bhx^{(2)}\leq\frac{1}{\widehat{\bfxi}^H\bH_i\widehat{\bfxi}}\widehat{\bfxi}^H\bH_i\widehat{\bfxi}=
1.\label{3.5}
\end{align}
In summary, the solution $\bhx$ generated from the proposed
randomization procedure is feasible for problem \eqref{P2}. Moreover, by the same argument as in Lemma \ref{lemma:2.2}, we have the following lemma.

 \begin{lemma}\label{lemma:3.0}
After Step S2 in the randomization procedure listed in Table
\ref{randomprocedure_2}, for all $ i\in \chI$, we have ${\hat \bfbeta}[i]\ge \frac{1}{M-Q+1}.$
\end{lemma}

We are now ready to prove the approximation bounds of the SDP for
problem \eqref{P2}, with $\epsilon\in(0,1)$.
First note that in the special case when $Q=M$, the problem \eqref{P2}
reduces to the following continuous QCQP
\begin{align}
\max_{\bw}&\quad\|\bw\|^2\nonumber\\
{\rm s.t.}&\quad {\bw^H{\bH_i}\bw}\le1,~i\in\cM\nonumber.
\end{align}
For this problem, it has been shown in \cite{Nemirovski99} that the
ratio between the optimal value of the original problem and its corresponding SDP relaxation
is bounded below by $\cO(1/\log M)$. Thus, we
only consider the case $Q\leq M-1$ in the following analysis. 

\begin{lemma}
Let $Q\leq M-1$ and $0<\epsilon\leq 1$. Suppose $\bhx$ is generated
by the randomization procedure given in Table
\ref{randomprocedure_2}. Then with probability $1$, $\bhx$ is well
defined and feasible for \eqref{1.4}. Moreover, for any $\gamma>0$
and $\hat{\mu}>0$,
\begin{align}
&{\rm Prob}\left( \hat{t}^2\geq \gamma, \ \|\bhx^{(2)}\|^2\geq \hat{\mu} {\rm Tr\,}[\bhX^{(2)}]\right)\nonumber\\
&\geq 1-\hat{\sigma}(\hat{\mu}, \hat{\gamma}),
\end{align}
with
\begin{equation}
\hat{\sigma}(\hat{\mu}, \hat{\gamma}) =\sum_{i\in \chI}\hat{r}_i\cdot e^{-\frac{1}{2\gamma}\cdot\frac{\epsilon}{\tilde{c}(\epsilon)}}+\sum_{i\in {\mathcal{M}\backslash{{\chI}}}}\hat{r}_i\cdot e^{-\frac{1}{2\gamma}}+\frac{0.968}{(1-\hat{\mu}/\gamma)}, \mbox{ if $\mathbb{F}=\mathbb{R}$},\label{lemmainequality3}
\end{equation}
and
\begin{equation}
\hat{\sigma}(\hat{\mu}, \hat{\gamma}) =\sum_{i\in \chI}\hat{r}_i\cdot e^{-\frac{1}{\gamma}\cdot\frac{\epsilon}{\tilde{c}(\epsilon) }}+\sum_{i\in {\mathcal{M}\backslash{{\chI}}}}\hat{r}_i\cdot e^{-\frac{1}{\gamma}}+\frac{2}{e(1-\hat{\mu}/\gamma)}, \mbox{ if $\mathbb{F}=\mathbb{C}$},\label{lemmainequality2}
\end{equation}
where $\hat{r}_i:=\min\{\rank (\bH_i), \rank
(\bhX^{(2)})\}$, $i\in \mathcal{M}$ and
\begin{equation}\label{eq:tildece}
\tilde{c}(\epsilon):= 1-\frac{1-\epsilon}{M-Q+1}.
\end{equation}.\label{lemma:3.1}
\end{lemma}

\begin{proof}
Since the density of
$\widehat{\bfxi}^H\bH_i\widehat{\bfxi}$ is continuous, thus the
probability $\widehat{\bfxi}^H\bH_i\widehat{\bfxi}=0$ is zero, which
implies that $\hat{t}$ computed by Step S4 in Table
\ref{randomprocedure_2} is well defined. The feasibility of $\bhx$
generated by the randomization procedure has been shown in Section \ref{Section:3.2}.
For any $\gamma>0$ and $\hat{\mu}>0$, we have
\begin{align}
&{\rm Prob}\left( \hat{t}^2\geq \gamma, \ \|\bhx^{(2)}\|^2\geq \hat{\mu} \trace[\bhX^{(2)}]\right)\nonumber\\
\geq &\mbox{Prob}\left( \hat{t}^2\geq \gamma, \ \|\hat{\bfxi}\|^2\geq \frac{\hat{\mu}}{\gamma} \trace[\bhX^{(2)}]\right)\nonumber\\
=&\mbox{Prob}\left(
\frac{\epsilon}{\hat{\bfxi}^H\bH_i\hat{\bfxi}}\geq \gamma, \forall
i\in \chI;\ \frac{1}{\widehat{\bfxi}^H\bH_i\widehat{\bfxi}}\geq
\gamma, \forall\; i\in {\mathcal{M}\setminus \chI};\
 \|\widehat{\bfxi}\|^2\geq \frac{\hat{\mu}}{\gamma} \trace[\bhX^{(2)}]\right)\nonumber\\
=&\mbox{Prob}\left( \widehat{\bfxi}^H\bH_i\widehat{\bfxi}\leq
\frac{\epsilon}{\gamma}, \forall\; i\in \chI;\
\widehat{\bfxi}^H\bH_i\widehat{\bfxi}\leq \frac{1}{\gamma}, \forall
i\in {\mathcal{M}\setminus \chI};\  \|\widehat{\bfxi}\|^2\geq
\frac{\hat{\mu}}{\gamma} \trace[\bhX^{(2)}]\right).\label{3.7}
\end{align}
By Lemma \ref{lemma:3.0}, we have
\begin{eqnarray}
\trace \left[\bH_i\bhX^{(2)}\right]&\leq& \hat\bfbeta[i]\cdot \epsilon+ (1-\hat\bfbeta[i])\cdot 1\nonumber\\
&=& 1+(\epsilon-1)\cdot \hat\bfbeta[i]\nonumber\\
&\le&1-\frac{1-\epsilon}{M-Q+1}:= \tilde{c}(\epsilon),\quad
\forall\; i\in \chI,\label{eqa:3.8}
\end{eqnarray}
where the first inequality follows from the feasibility of  ${(\hat\bfbeta, \bhX^{(2)})}$. On the other hand, we have
\begin{align}
\trace \left[\bH_i\bhX^{(2)}\right]&\leq \hat\bfbeta[i]\cdot \epsilon+ (1-\hat\bfbeta[i])\cdot 1\le 1, \quad  \forall\; i\in {\mathcal{M}\backslash{{\chI}}}.\label{eqa:3.9}
\end{align}

Note that
$\mathbb{E}\left(\widehat{\bfxi}\widehat{\bfxi}^{T}\right)=\bhX^{(2)}$
and
$\mathbb{E}\left(\widehat{\bfxi}^H\bH_i\widehat{\bfxi}\right)=\trace
[\bH_i\bhX^{(2)}]$. By combining \eqref{3.7}, \eqref{eqa:3.8} and \eqref{eqa:3.9}, we
have
\begin{align}
&{\rm Prob}\left( \hat{t}^2\geq \gamma, \ \|\bhx^{(2)}\|^2\geq \hat{\mu} \trace[\bhX^{(2)}]\right)\nonumber\\
\geq &\mbox{Prob}\left( \widehat{\bfxi}^H\bH_i\widehat{\bfxi}\leq \frac{1}{\gamma}\cdot\frac{\epsilon}{\tilde{c}(\epsilon) }\cdot\trace[\bH_i\bhX^{(2)}],\
\forall\; i\in \chI \right.;\nonumber\\
&\quad\quad\ \ \left.\widehat{\bfxi}^H\bH_i\widehat{\bfxi}\leq \frac{1}{\gamma}\trace[\bH_i\bhX^{(2)}],\
\forall\; i\in {\mathcal{M}\setminus \chI};\
\|\widehat{\bfxi}\|^2\geq \frac{\hat{\mu}}{\gamma} \trace[\bhX^{(2)}]\right)\nonumber\\
\geq & 1- \sum_{i\in
\chI}\mbox{Prob}\left(\widehat{\bfxi}^H\bH_i\widehat{\bfxi}>
\frac{1}{\gamma}\cdot\frac{\epsilon}{\tilde{c}(\epsilon) }\cdot\trace[\bH_i\bhX^{(2)}]\right)\nonumber\\
&\quad -\sum_{i\in
\mathcal{M}\setminus \chI}\mbox{Prob}
\left(\widehat{\bfxi}^H\bH_i\widehat{\bfxi}> \frac{1}{\gamma}\trace[\bH_i\bhX^{(2)}]\right)
-\mbox{Prob}\left(\|\widehat{\bfxi}\|^2< \frac{\hat{\mu}}{\gamma} \trace[\bhX^{(2)}]\right).
\end{align}
When $\mathbb{F}=\mathbb{C}$, we have
\begin{align}
&{\rm Prob}\left( \hat{t}^2\geq \gamma, \ \|\bhx^{(2)}\|^2\geq \hat{\mu} \trace[\bhX^{(2)}]\right)\nonumber\\
&\geq  1-\sum_{i\in \chI}\hat{r}_i \cdot e^{-\frac{1}{\gamma}\cdot\frac{\epsilon}{\tilde{c}(\epsilon) }}-\sum_{i\in {\mathcal{M}\backslash{{\chI}}}}\hat{r}_i\cdot e^{-\frac{1}{\gamma}}-\mbox{Prob}\left(\|\hat{\bfxi}\|^2< \frac{\hat{\mu}}{\gamma} \trace[\bhX^{(2)}]\right)\label{3.9}\\
&\geq  1-\sum_{i\in \chI}\hat{r}_i\cdot e^{-\frac{1}{\gamma}\cdot\frac{\epsilon}{\tilde{c}(\epsilon)}}-\sum_{i\in {\mathcal{M}\backslash{{\chI}}}}\hat{r}_i\cdot e^{-\frac{1}{\gamma}}-\frac{2}{e(1-\hat{\mu}/\gamma)},\label{3.10}
\end{align}
where \eqref{3.9} is from the Lemma 5 in
\cite{Luo07approximationbounds}, and the last inequality \eqref{3.10} can be
verified by using Markov inequality and the property of Weibull
distribution (the details can be found in the proof of \cite[Theorem
5]{Luo07approximationbounds}). When $\mathbb{F}=\mathbb{R}$, the conclusion can be similarly proved by utilizing the results stated in \cite[Page 22]{Luo07approximationbounds}.
\end{proof}

\begin{theorem}\label{max_real}
When $\mathbb{F}=\mathbb{R}$, for real QCQP problem \eqref{1.4} and its SDP relaxation \eqref{problemSDP}, we have
$$v^{\max}_{{\rm QP}}\geq \frac{\epsilon}{\tilde{c}(\epsilon)} \cdot\frac{1}{200\ln (50K)}v^{\max}_{{\rm SDP}},$$
where $K=\sum_{i=1}^{M} \min \{\rank (\bH_i), \sqrt{2M}\}$ and $\tilde{c}(\epsilon)$ is given by \eqref{eq:tildece}.
\end{theorem}
\begin{proof}
Note that $0<\frac{\epsilon}{\tilde{c}(\epsilon)}\le 1$, by Lemma
\ref{lemma:3.1}, we have
\begin{align}
&{\rm Prob}\left( \hat{t}^2\geq \gamma, \ \|\bhx^{(2)}\|^2\geq \hat{\mu} \trace[\bhX^{(2)}]\right)\nonumber\\
&\ge 1-\sum_{i\in \chI}\hat{r}_i\cdot e^{-\frac{1}{\gamma}\cdot\frac{\epsilon}{\tilde{c}(\epsilon) }}-\sum_{i\in {\mathcal{M}\backslash{{\chI}}}}\hat{r}_i\cdot e^{-\frac{1}{\gamma}\cdot\frac{\epsilon}{\tilde{c}(\epsilon) }}-\frac{0.968}{(1-\hat{\mu}/\gamma)}.\label{3.13}
\end{align}
On the other hand, by applying a suitable rank reduction procedure if necessary, we can
assume that $\rank (\bhX^{(2)})\leq \sqrt{2M}$; see
\cite[Section 5]{Zhang04complexquadratic}.  We apply the
randomization procedure to $\bhX^{(2)}$. From \eqref{3.13} we have, for any $\gamma>0$ and $\hat{\mu}>0$,
\begin{align}
&{\rm Prob}\left( \hat{t}^2\geq \gamma, \ \|\bhx^{(2)}\|^2\geq \hat{\mu} \trace[\bhX^{(2)}]\right)\nonumber\\
&\ge 1-K\cdot e^{-\frac{1}{\gamma}\cdot\frac{\epsilon}{\tilde{c}(\epsilon) }}-\frac{0.968}{(1-\hat{\mu}/\gamma)}.\label{3.14}
\end{align}
Setting $$\gamma=\frac{\epsilon}{\tilde{c}(\epsilon)} \cdot\frac{1}{2\ln (50K)} ,\quad
\hat{\mu}=0.01\gamma,$$ we have that the right hand side of \eqref{3.14}
 is larger than $
1-\frac{1}{50}-\frac{0.968}{0.99}=0.00222\ldots$, which
then proves the desired bound.
\end{proof}

\begin{theorem}\label{max_complex}
When $\mathbb{F}=\mathbb{C}$, for complex QCQP problem \eqref{1.4} and its SDP relaxation \eqref{problemSDP}, we have
$$v^{\max}_{{\rm QP}}\geq \frac{\epsilon}{ \tilde{c}(\epsilon)}\cdot \frac{1}{4\ln (100K)} v^{\max}_{{\rm SDP}},$$
where $\tilde{c}(\epsilon)$ is defined by \eqref{eq:tildece}.
\end{theorem}

\begin{proof}
By applying a suitable rank reduction procedure if necessary, we can
assume that $\rank (\bhX^{(2)})\leq \sqrt{M}$; see
\cite[Section 5]{Zhang04complexquadratic}.  We apply the
randomization procedure to $\bhX^{(2)}$. By using Lemma
\ref{lemma:3.1}, we have, for any $\gamma>0$ and $\hat{\mu}>0$,
\begin{align}
&{\rm Prob}\left( \hat{t}^2\geq \gamma, \ \|\bhx^{(2)}\|^2\geq \hat{\mu} \trace[\bhX^{(2)}]\right)\nonumber\\
&\geq 1-K\cdot e^{-\frac{1}{\gamma}\cdot\frac{\epsilon}{\tilde{c}(\epsilon)}}-\frac{2}{e(1-\hat{\mu}/\gamma)}.\label{3.11}
\end{align}
The proof is similar to that of \eqref{3.14}. Setting $$\gamma=\frac{\epsilon}{\tilde{c}(\epsilon) }\cdot \frac{1}{\ln (100K)},\quad
\hat{\mu}=\frac{1}{4}\gamma,$$ we have that the right hand side of
\eqref{3.11} is larger than $
1-\frac{1}{100}-\frac{8}{3e}=0.00899\ldots$, which
then proves the desired bound.
\end{proof}

Theorem \ref{max_real} and Theorem \ref{max_complex} show that for fixed $0<\epsilon<1$, the approximation ratio for the maximization model is the same order of $\cO(1/\log M)$. It is tight for general $Q$, since in a special case, i.e., $Q=M$,  $\cO(1/\log M)$ has been proved to be a tight bound \cite{Nemirovski99}.

We note that for the maximization model, it is possible to use
different rounding techniques to obtain $\bx^{(2)}$. For example we
can decompose the matrix $\bhX^{(2)}$ using the techniques proposed
in \cite{Nemirovski99} and \cite{He2008}. Moreover, if the decomposition techniques are used, through similar analysis (as in \cite{Nemirovski99}), we also can obtain the following approximation ratio
\begin{equation}v^{\max}_{{\rm QP}}\geq \frac{\epsilon}{\tilde{c}(\epsilon) }\cdot\frac{1}{2\ln(2M\mu)} v^{\max}_{{\rm SDP}},\label{max_ref21}\end{equation}
where $\mu$ is the same as defined in \cite{Nemirovski99}. This ratio is sharper than our result given in Theorem \ref{max_real}. However, it seems that the algorithm that proposed in Table \ref{randomprocedure_2} and the corresponding analysis are simpler than that of \cite{Nemirovski99}. That's the reason we do not show the details of the above result.  Nevertheless, both these two approaches lead to an approximation bound of the same order in terms of $M$.

As pointed out by a referee, there exist different relaxations  of \eqref{P2}, one of which is given below.
\begin{align}
\max &\quad\|\bw\|^2\nonumber\\
{\rm s.t.}&\quad \bw^H \bH_i \bw\le\epsilon,~i=1,\cdots, M.\label{relaxation_max}
\end{align}
Denote the optimal objective value to be $v^*$. The SDP relaxation of \eqref{relaxation_max} is given by
\begin{align}
\max &\quad \trace{(\bX)}\nonumber\\
{\rm s.t.}&\quad \trace{(\bH_i \bX)}\le \epsilon,~i=1,\cdots, M.\label{SDP_max}
\end{align}
Denote the optimal objective value of \eqref{SDP_max} to be $v_{\rm SDP}. $ By using Ben-Tal et al 's result \cite{Bental02}, we can find  a feasible solution $x$ such that
\begin{align}
&\bx^H \bH_i \bx\le \epsilon,~i=1,\cdots, M\nonumber\\
& \|x\|^2\geq O\left(\frac{1}{\ln M}\right)v_{\rm SDP}\geq  O\left(\frac{\epsilon}{\ln M}\right)v^*.\label{result_max}
\end{align}
The analysis of this method is much simpler. However, we prefer the relaxation \eqref{problemSDP} due to two reasons.
Firstly, the approximation bound for the SDP relaxation problem \eqref{problemSDP} is a little sharper than the one obtained above. It is easy to see that the coefficient in the ratio for  \eqref{problemSDP} is $\frac{\epsilon}{\tilde{c}(\epsilon)}$, as shown in Theorem \ref{max_real}, Theorem \ref{max_complex} and \eqref{max_ref21},  which is larger than  $\epsilon$ in \eqref{result_max}. Moreover, it is obvious that the optimal value of the SDP relaxation problem \eqref{problemSDP}  is always larger than that of \eqref{SDP_max}, i.e.,
$$v^{\max}_{{\rm SDP}} \ge  v_{\rm SDP}.$$
Thus, although the two ratios are of the same order in terms of $M$ and $\epsilon$, the approximation bound for \eqref{problemSDP} is slightly sharper than that of \eqref{SDP_max}.
Secondly,  the SDP relaxation problem \eqref{SDP_max} is less attractive, because it requires all $M$ constraints to be less than $\epsilon$, which is too restrictive especially for small $\epsilon$.  This observation has also been confirmed by numerical experiments.

{\bf Remark.} We mention that all our theoretic results for both the
minimization and the maximization models could be extended to
general strictly convex quadratic objectives, by using a simple
variables substitution. Specifically, if the objective is $ \bw^H
\bA\bw$ with $\bA\succ 0$, by letting $\bhw=\bV\bw$ and
$\bhH_i=\bV^{-H}\bH_i\bV^{-1}$ where $\bV\succ 0$ satisfies
$\bV^{H}\bV=A$, then the corresponding models with the new variables
$\bhw$ and the new constraint matrices $\bhH_i$ are the same with
our current models.


 \section{Numerical Experiments}
 \begin{figure}[t]
\caption{Upper bound on $v^{\min}_{{\rm QP}}/v^{\min}_{{\rm SDP}}$
for $M=8$, $Q=6$, $N=8$, $300$ realizations of real Gaussian i.i.d.
$\bh_i$ for $i=1,\cdots, M$ in the real case.} \label{fig_real1}
\includegraphics[width=\textwidth]{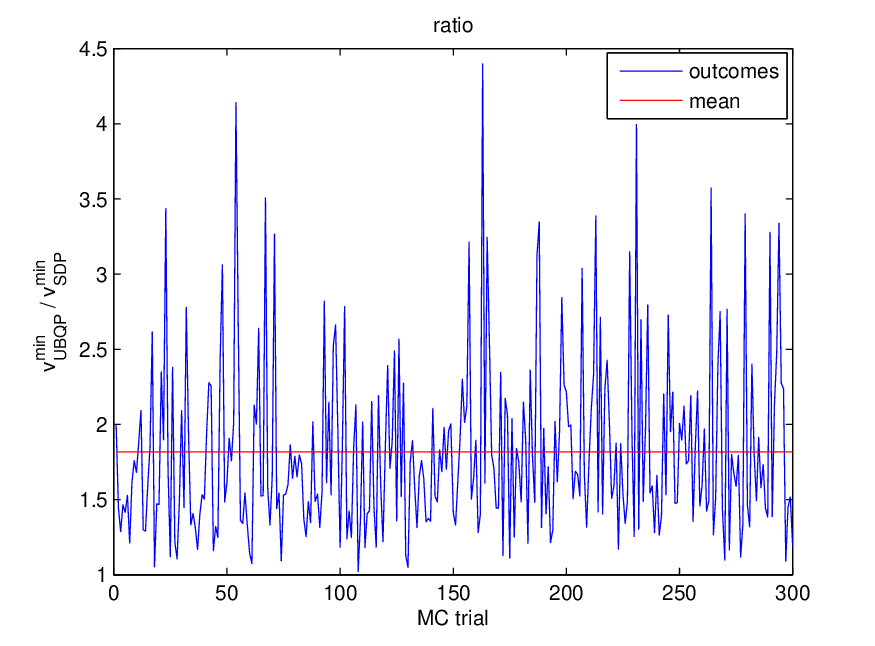}
\end{figure}
\begin{figure}[t]
\caption{Histogram of the outcomes in Figure \ref{fig_real1}.}
\label{fig_real2}
\includegraphics[width=\textwidth]{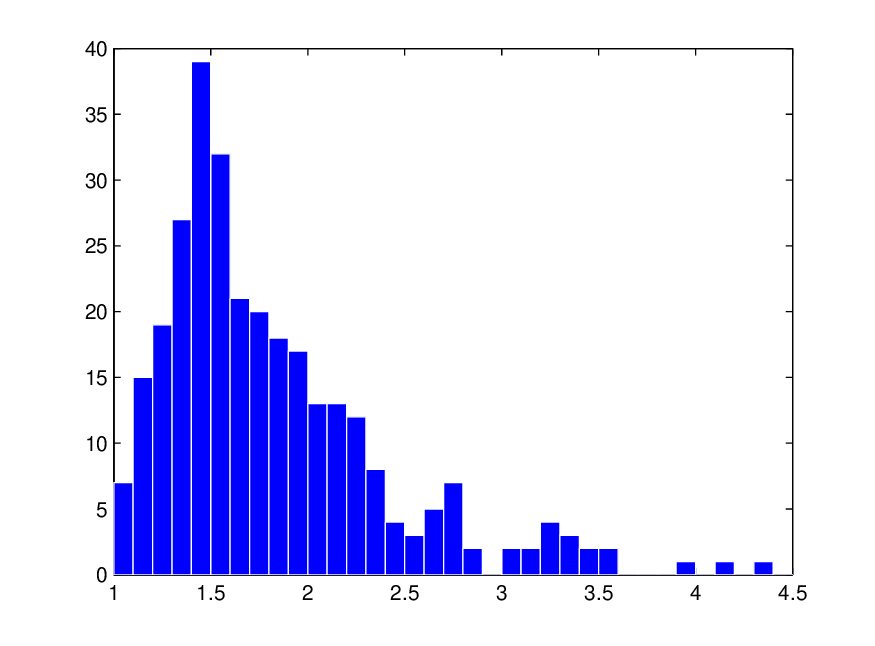}
\end{figure}
\begin{table}
\caption{Mean and standard deviation of the approximation ratio over
$300$ independent realizations of real Gaussian i.i.d. $\bh_i$
$(i=1,\cdots, M)$, when $\mathbb{F}=\mathbb{R}$.}
\begin{center} \footnotesize
\begin{tabular}{|l|l|ccc|ccc|} \hline
\multirow{2}{*}{$M$} &  \multirow{2}{*}{$Q$}  &
\multicolumn{3}{c}{$N=4$} & \multicolumn{3}{|c|}{$N=8$}\\
\cline{3-8} & &max & mean & Std &max & mean & Std \\ \hline
\multirow{3}{*}{$M=8$} &  $Q=M/4$ & 3.7394 & 2.0348 & 0.2266 & 4.3387 &  2.0392 & 0.2948 \\
&  $Q=M/2$ & 3.9420 & 1.7972 & 0.1828 &3.5232 & 1.7378 & 0.1475\\
&  $Q=3M/4$&4.6973 & 1.7863 & 0.3921 &4.5721 & 1.8130 & 0.3428\\ \hline
\multirow{3}{*}{$M=12$} &  $Q=M/4$ &4.9450& 2.2191 & 0.2451 & 3.9625 &  2.1710 & 0.2304\\
&  $Q=M/2$ &5.8068 & 2.0639 & 0.4564 & 4.3483 & 2.0204 & 0.3241\\
&  $Q=3M/4$ &7.7829 & 2.5970 & 1.3075 & 9.7150 & 2.8277  &  1.9578\\ \hline
\multirow{3}{*}{$M=16$} &  $Q=M/4$ &4.2703& 2.2977 & 0.2410 &4.2980&  2.2117 & 0.1972\\
&  $Q=M/2$ &7.3115 & 2.4463 & 0.9348 & 7.8240 & 2.4166 & 1.1345\\
&  $Q=3M/4$ &10.715 & 3.2272 & 2.3823 & 10.621& 3.7786  & 2.8760\\ \hline
\end{tabular}
\end{center}
\label{min_real}
\end{table}
\begin{table}
\caption{Mean and standard deviation of upper bound ratio over $300$ independent realizations of real Gaussian i.i.d. $\bh_i$ $(i=1,\cdots, M)$, when $\mathbb{F}=\mathbb{C}$.}
\begin{center} \footnotesize
\begin{tabular}{|l|l|ccc|ccc|} \hline
\multirow{2}{*}{$M$} &  \multirow{2}{*}{$Q$}  &
\multicolumn{3}{c}{$N=4$} & \multicolumn{3}{|c|}{$N=8$}\\
\cline{3-8} & &max & mean & Std &max & mean & Std \\ \hline
\multirow{3}{*}{$M=8$} &  $Q=M/4$ & 4.8049 & 2.3720 & 0.2790 & 4.3579 &  2.4239 & 0.2757 \\
&  $Q=M/2$ & 3.7344 & 1.9308 & 0.1443 &3.4477 & 1.9243 & 0.1477\\
&  $Q=3M/4$&2.7549 & 1.5812 & 0.0769 &2.4477 & 1.5860 & 0.0818\\ \hline
\multirow{3}{*}{$M=12$} &  $Q=M/4$ &4.0557& 2.4986 & 0.1938 & 3.7851 &  2.4657 & 0.1998\\
&  $Q=M/2$ &3.2911 & 2.0301 & 0.1483 & 3.2867 & 2.0567 & 0.1190\\
&  $Q=3M/4$ &2.6007 & 1.6451 & 0.0800 & 3.1609 & 1.6693 &  0.0860\\ \hline
\multirow{3}{*}{$M=16$} &  $Q=M/4$ &3.8170& 2.5778 & 0.1647 &4.2616&  2.5852 & 0.1757\\
&  $Q=M/2$ &3.6268 & 2.0908 & 0.0932 & 3.7761 & 2.0729 & 0.1065\\
&  $Q=3M/4$ &2.9218 & 1.8024 & 0.1044 & 3.6056& 1.8344  & 0.1432\\ \hline
\end{tabular}
\end{center}
\label{min_complex}
\end{table}
\begin{figure}[t]
\caption{Upper bound on $v^{\min}_{{\rm QP}}/v^{\min}_{{\rm SDP}}$
for $M=8$, $Q=6$, $N=8$, $300$ realizations of real Gaussian i.i.d.
$\bh_i$ for $i=1,\cdots, M$ in the complex case.}
\label{fig_complex1}
\includegraphics[width=\textwidth]{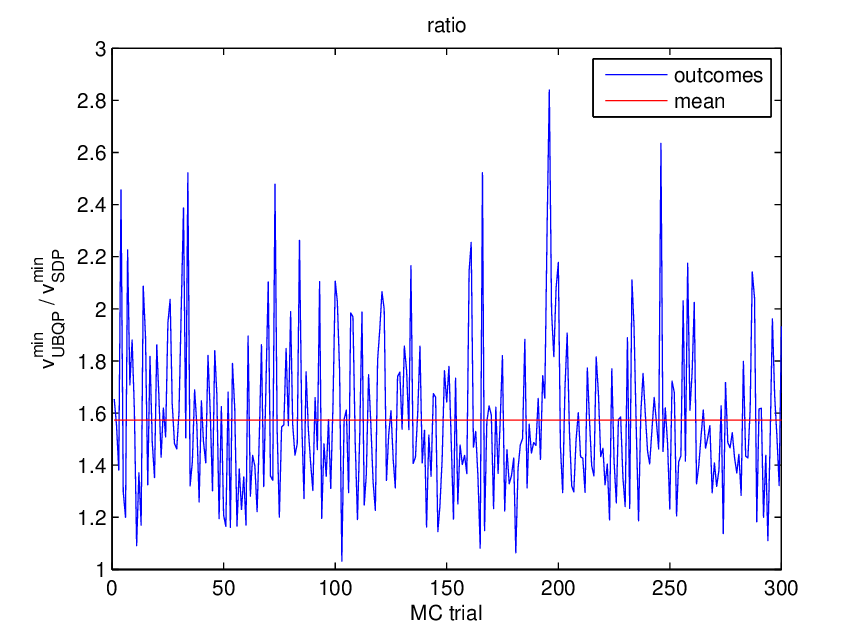}
\end{figure}
\begin{figure}[t]
\caption{Histogram of the outcomes in Figure \ref{fig_complex1}.}
\label{fig_complex2}
\includegraphics[width=\textwidth]{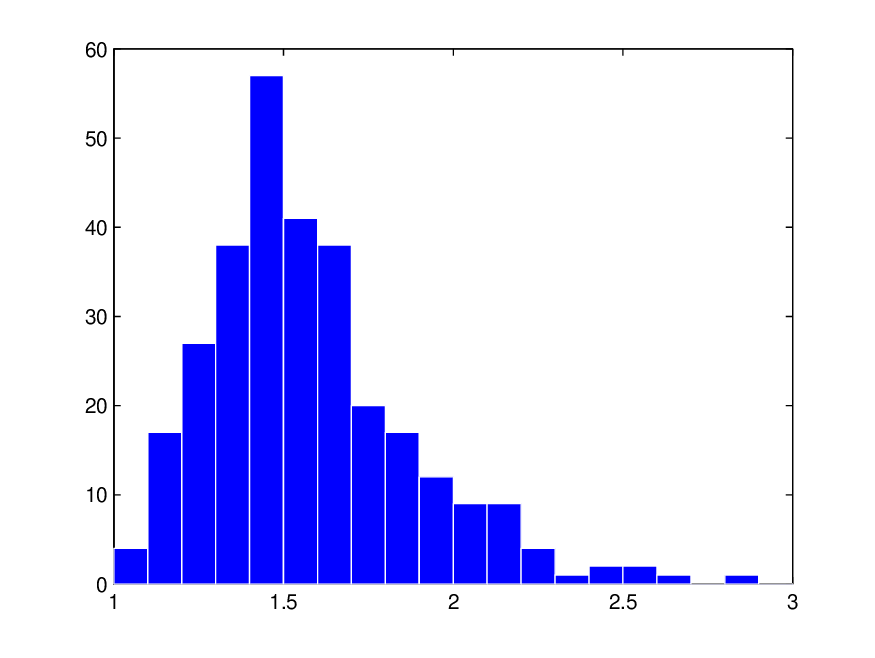}
\end{figure}
%

 While theoretical worst-case analysis is very useful, empirical analysis
 of the ratio $v^{\min}_{{\rm QP}}/v^{\min}_{{\rm SDP}}$ and $v^{\max}_{{\rm QP}}/v^{\max}_{{\rm SDP}}$ through simulations can often provide valuable insights into the true efficacy of the relaxation methods.  Throughout this section, we generate
 the data matrix $\bH_i$ by using $\bH_i=\bh_i\bh_i^H$ $(i=1,\cdots, M)$, with randomly generated vectors $\bh_i$. The SDP relaxation problems are all solved by CVX \cite{CVX2011}.

For the minimization model, we test the proposed procedure listed
in Table \ref{tableRandomize} for $\epsilon=0$ and different choices of $M$, $Q$ and
$N$. The performance of the algorithm with other choices of $\epsilon$ are similar. The Step S3 and Step S4 are repeated by $T=1000$ independent
trials, and the solution generated by $k$th trial is denoted by
$(\bx^{(2)})^k$. Let
$$v^{\min}_{{\rm UBQP}}:=\min_{k=1,\cdots, T}\|(\bx^{(2)})^k\|^2.$$
Clearly $v^{\min}_{{\rm UBQP}}$ is an upper bound for
$v^{\min}_{{\rm QP}}$, as a result, $v^{\min}_{{\rm
UBQP}}/v^{\min}_{{\rm SDP}}$ can be used as an upper bound of the
true approximation ratio (which is difficult to obtain).

Table \ref{min_real} shows the average ratio (mean) of
$v^{\min}_{{\rm UBQP}}/v^{\min}_{{\rm SDP}}$ over 300 independent
realizations of i.i.d. real-valued Gaussian $\bh_i$, $(i=1,\cdots,
M)$ for several combinations of $M$, $Q$ and $N$. The maximum value
(max) and the standard deviation (Std) of $v^{\min}_{{\rm
UBQP}}/v^{\min}_{{\rm SDP}}$ over 300 independent realizations are
also shown in Table \ref{min_real}. Table \ref{min_complex} shows
the corresponding average value, maximum value and the standard
deviation of $v^{\min}_{{\rm UBQP}}/v^{\min}_{{\rm SDP}}$ for
$\mathbb{F}=\mathbb{C}$. These results are significantly better than
what is predicted by our worst-case analysis. In all test examples,
the average values of $v^{\min}_{{\rm UBQP}}/v^{\min}_{{\rm SDP}}$
are lower than $4$ (resp. lower than $3$) when
$\mathbb{F}=\mathbb{R}$ (resp. when $\mathbb{F}=\mathbb{C}$).

Figure \ref{fig_real1} plots $v^{\min}_{{\rm UBQP}}/v^{\min}_{{\rm SDP}}$
for 300 independent realizations of i.i.d. real valued Gaussian
$\bh_i$ ($i=1,\cdots, M$) for $M=8$, $Q=6$ and $N=8$. Figure
\ref{fig_real2} shows the corresponding histogram. Figure
\ref{fig_complex1} and Figure \ref{fig_complex2} show the
corresponding results for i.i.d complex-valued circular Gaussian
$\bh_i$ ($i=1,\cdots,M$). Both the mean and the maximum of the upper
bound $v^{\min}_{{\rm UBQP}}/v^{\min}_{{\rm SDP}}$ are lower in the
complex case.

Moreover, our numerical results also corroborates well with our
theoretic analysis. First, the upper bound of the approximation
ratio is independent of the dimension of $\bw$: the results vary
only slightly for $N=4$ and $N=8$ in both real and complex case.
Second, from Table \ref{min_real}, it can be shown that for fixed
$M$, the maximum value of $v^{\min}_{{\rm UBQP}}/v^{\min}_{{\rm
SDP}}$ over 300 independent trials grows as $Q$ increases in all
test examples except $M=8$ and $N=8$. It corresponds to the result
in Theorem \ref{Approximationratio_R}. While in Table
\ref{min_complex}, for fixed $M$, the maximum value of
$v^{\min}_{{\rm UBQP}}/v^{\min}_{{\rm SDP}}$ over 300 independent
trials becomes smaller as $Q$ increases in all test examples, which is
also consistent with the theoretic result listed in Theorem
\ref{thmApproximationStep2complex}.

The above empirical analysis complements our theoretic worst-case
analysis of the performance of SDP relaxation for the class of mixed integer QCQP
problems considered herein.

 \section{Conclusion and Discussion}
 Motivated by important emerging applications in transmit beamforming
 for joint physical layer multicasting and admission control in wireless networks,
 this paper proposes new SDP relaxation techniques for two classes of nonconvex quadratic optimization
 problems with mixed binary and continuous variables. It is shown that these efficient techniques (polynomial time)
 are guaranteed to provide high quality approximate solutions with a finite approximation ratio that is independent of problem dimension and data matrices. This work extends the existing SDP
relaxation techniques for continuous nonconvex QCQPs.
 Our theoretic analysis provides useful insights on the effectiveness of the new SDP relaxation techniques
 for this class of mixed integer QCQP problems.


It should be pointed out that our worst-case analysis of SDP relaxation performance
is based on a certain structure of the discrete variables in the mixed integer QCQPs. Can we
extend the SDP relaxation techniques and the corresponding analysis to more general MBQCQP problems? For the
maximization model, such extension is not possible, as the
approximation ratio for SDP relaxation can be zero for certain
special instances of the problem (see e.g., Example 3.1). For the
minimization model, the following counter example suggests that this
is also impossible.

%

{\bf Example 5.1. } Consider the following general MBQCQP problem in
minimization form:
\begin{align}
\min_{\bx \in \mathbb{R}^{n}}&\quad\|\bx\|^2\nonumber\\
{\rm s.t.}&\quad \bx^H \bD_i \bx\ge 1,~i=1,\cdots,M\label{Generalmodel}\\
&\quad (\bx[j])^2=1,~j=1,\cdots,l\nonumber
\end{align}
with $l\leq n$ and $\bD_i$, $i=1,\cdots, M$ being arbitrarily positive semidefinite matrices. The
SDP relaxation for this problem can be expressed as
 \begin{align}
\min &\quad \trace [\bX]\nonumber\\
{\rm s.t.}&\quad \trace [\bD_i \bX]\ge 1,~i=1,\cdots,M\label{Generalmodel_SDP}\\
&\quad \bX[j,j]=1,~j=1,\cdots,l\nonumber\\
& \quad \bX \succeq 0\nonumber.
\end{align}
We show in the following that SDP relaxation
\eqref{Generalmodel_SDP} can be very loose for this general case.
Let $M=2$, $l=1$, $n=2$, and let
\[\bD_1=\left[\begin{array}{cc}1-\epsilon& \sqrt{\epsilon(1-\epsilon)}\\\sqrt{\epsilon(1-\epsilon)}& \epsilon\end{array}\right],\qquad
\bD_2=\left[\begin{array}{cc}1&-\sqrt{\frac{\epsilon}{2}}\\-\sqrt{\frac{\epsilon}{2}}&\epsilon\end{array}\right]\]
where $\epsilon$ is some positive constant. It is relatively easy to show
that the optimal solution ${\bar \bx}$ for problem \eqref{Generalmodel} is give by
$$\bar{\bx}[1]=1, \quad \bar{\bx}[2]=\min \left\{\sqrt{\frac{2}{\epsilon}}, \ \frac{\sqrt{1-\epsilon}+1}{\sqrt{\epsilon}}\right\}.$$
Thus, when $\epsilon \rightarrow 0$, we have
$\bar{\bx}[2]\rightarrow +\infty$, implying $\|\bar{\bx}\|^2\rightarrow +\infty$.
However, it can be easily checked that $\bX=\bI$ is a feasible solution for the problem
\eqref{Generalmodel_SDP}.
Therefore, the optimal solution ${\bar \bX}$ for \eqref{Generalmodel_SDP}
should satisfy
$$\trace [{\bar \bX}] \le\trace[\bI]= 2.$$
This example shows that, for a general MBQCQP problem in minimization
form, the approximation ratio can be
arbitrarily large, i.e., the worst-case approximation ratio as
stated in \eqref{defratio} is $\mu=+\infty$.

As a final remark, we mention that the the approximation bounds obtained in this work are due to the special structures of the problems under consideration. For example, the constraints on the discrete variables are relatively simple. Moreover, the discrete variables and the continuous variables are separable, e.g., there are no cross terms between the discrete and the continuous variables, either in the constraints or the objective. These nice  properties allow us to design algorithms that can separately deal with the discrete and the continuous variables.


\section*{Acknowledgments}
We are very grateful to an anonymous referee for his/her insightful comments which have helped to improve the paper.

\end{document}